\documentclass[reqno,11pt,a4paper]{amsart}
\usepackage[draft=false]{hyperref}
\usepackage{microtype}
\usepackage{url}
\usepackage{amsmath,amssymb,amsfonts}
   \DeclareMathOperator{\Id}{Id}
   \DeclareMathOperator{\e}{e}
   \DeclareMathOperator{\Lip}{Lip}

\usepackage{amsthm}
   \newtheorem{theorem}{Theorem}[section]
   \newtheorem{lemma}[theorem]{Lemma}
   
   \newtheorem{example}[theorem]{Example}
   \newtheorem{corollary}[theorem]{Corollary}
\usepackage{dsfont,mathrsfs}
   \newcommand{\N}{\mathds{N}}
   \newcommand{\Z}{\mathds{Z}}
   \newcommand{\R}{\mathds{R}}
   \newcommand{\T}{\mathbb{T}}

   \newcommand{\sF}{\mathscr{F}}
   
   \newcommand{\cG}{\mathcal{D}}
   \newcommand{\cH}{\mathcal{C}}
   \newcommand{\cV}{\Gamma}
   \newcommand{\cX}{\mathfrak{U}}
   \newcommand{\cL}{\mathfrak{D}}
   \newcommand{\fJ}{\mathfrak{C}}

   \newcommand{\JJ}{C}   
   \newcommand{\LL}{D}      
   \newcommand{\TT}{U}   
\usepackage{enumitem}
   \newcommand{\lb}{label}
   \newcommand{\lm}{leftmargin}

\newcommand{\eps}{\varepsilon}
\newcommand{\lbd}{\lambda}
\newcommand{\sgm}{\sigma}

\newcommand{\w}{\omega}
\newcommand{\W}{\Omega}

\newcommand{\prts}[1]{\!\left(#1\right)}
\newcommand{\prtsr}[1]{\left[#1\right]}
\newcommand{\abs}[1]{\left|#1\right|}
\newcommand{\norm}[1]{\left\|#1\right\|}
\newcommand{\set}[1]{\left\{#1\right\}}

\newcommand{\maxs}[1]{\max\set{#1}}
\newcommand{\mins}[1]{\min\set{#1}}
\newcommand{\sups}[1]{\sup\set{#1}}

\newcommand{\x}{\times}

\newcommand{\pfrac}[2]{\prts{\dfrac{#1}{#2}}}

\newcommand{\dsum}{\displaystyle\sum}
\newcommand{\dint}{\displaystyle\int}

\newcommand{\dlim}{\displaystyle\lim}
\newcommand{\dr}{\, dr}
\newcommand{\ds}{\, ds}

\newcommand{\cP}{\mathcal{P}}

\newcommand{\cc}{{{c}}}
\newcommand{\sss}{{{s}}}
\newcommand{\uu}{{{u}}}
\newcommand{\cs}{{{\overline{c}}}}
\newcommand{\cu}{{{\underline{c}}}}

\renewcommand{\phi}{\varphi}

\newcommand{\fthr}{f_{\theta^r \w}\prts{h_{r,\w}\prts{\xi}, \phi_{\theta^r\w}\prts{h_{r,\w}\prts{\xi}}}}
\renewcommand{\ge}{\geqslant}
\renewcommand{\geq}{\geqslant}
\renewcommand{\le}{\leqslant}
\renewcommand{\leq}{\leqslant}
\begin{document}
\title[Center manifolds for RDS]{Center manifolds for Random Dynamical Systems with generalized trichotomies}
\author[Ant\'onio J. G. Bento]{Ant\'onio J. G. Bento}
   \address{
      Ant\'onio J. G. Bento\\
      Centro de Matem\'atica e Aplica\c{c}\~oes and Departamento de Matem\'atica\\
      Universidade da Beira Interior\\
      6201-001 Covilh\~a\\
      Portugal}
   \email{bento@ubi.pt}
\author{Helder Vilarinho}
   \address{
      Helder Vilarinho\\
      Centro de Matem\'atica e Aplica\c{c}\~oes and Departamento de Matem\'atica\\
      Universidade da Beira Interior\\
      6201-001 Covilh\~a\\
      Portugal}
   \email{helder@ubi.pt}
   \urladdr{www.helder.ubi.pt}
\date{\today}
\begin{abstract}
   For small perturbations of linear Random Dynamical Systems evolving on a Banach space and exhibiting a generalized form of trichotomy, we prove the existence of invariant center manifolds, both in continuous and discrete-time. Furthermore, we provide several illustrative examples.
\end{abstract}
\maketitle
\section{Introduction}

   The theory of center manifolds plays a crucial role in stability and bifurcation theory because it often decreases the dimension of the state space (see \cite{Carr-book-1981, Henry-book-1981, Guckenheimer_Holmes-book-1990, Hale_Kocak-book-1991, Haragus_Iooss-book-2011}). The work on center manifolds goes back to the sixties with the papers by Pliss~\cite{Pliss-IANSSM-1964} and by Kelley~\cite{Kelley-JMAA-1967, Kelley-JDE-1967}, after which different results about this subject were established by several authors. 
   For center manifolds of autonomous differential equations we recommend the surveys by Vanderbauwhede~\cite{Vanderbauwhede-DR-1989} (see also Vanderbauwhede and Gils~\cite{Vanderbauwhede_Gils-JFA-1987}) in finite dimension and by Vanderbauwhede and Iooss~\cite{Vanderbauwhede_Iooss-DR-1992} in infinite dimension. 
   For the nonautonomous case we recommend the survey by Aulbach and Wanner \cite{Aulbach_Wanner-book-1996}. We also recommend~\cite{Chow_Liu_Yi-TAMS-2000, Chow_Liu_Yi-JDE-2000} for finite dimension and~\cite{Sijbrand-TAMS-1985, Mielke-JDE-1986, Chow_Lu-PRSE-1988,Chow_Lu-JDE-1988, Chicone_Latushkin-JDE-1997} for infinite dimension.
   
   An important instrument to obtain center manifolds is the concept of trichotomy.    
   The (uniform) exponential trichotomies were introduced, independently, by Sacker and Sell \cite{Sacker_Sell-JDE-1976-22-(497-522)}, Aulbach \cite{Aulbach-NATMA-1982} and Elaydi and Hájek \cite{Elaydi_Hajek-JMAA-1988}. This notion is motivated by the idea of (uniform) exponential dichotomy that started in the thirties with Perron~\text{\cite{Perron-MZ-1929,Perron-MZ-1930}}. 
   
   After that several generalizations have emerged. Fenner and Pinto~\cite{Fenner_Pinto-JMAA-1997} presented the $(h,k)$-trichotomies that use nonexponential growth rates and Barreira and Valls~\text{\cite{Barreira_Valls-JMPA-2005,Barreira_Valls-ETDS-2006}} presented nonuniform exponential trichotomies that also takes into account the initial time. Later, Barreira and Valls~\text{\cite{Barreira_Valls-CPAA-2010,Barreira_Valls-SPJMS-2011}} introduced the $\rho$-nonuniform exponential trichotomies that are nonuniform and nonexponential, but do not include the $(h,k)$-trichotomies.

   In~\cite{Bento_Costa-EJQTDE-2017,Bento-CLM-2022}, a  general type of trichotomies was introduced, for linear differential equations and linear difference equations, respectively. This new framework  contains as special cases the notions of trichotomies mentioned above and also contains additional new cases (the case of dichotomies was done in \cite{Bento_Silva-BSM-2014,Bento_Silva-PM-2016}).

The theory of center manifolds has been expanded to include dynamical systems exhibiting randomness. In this work, we focus on Random Dynamical Systems (RDS), which can be generated, for instance, by random or stochastic differential equations.

Several works address various types of invariant manifolds - center, stable, unstable, and inertial - both locally and globally. These studies encompass a range of spaces, from finite dimensions, such as Euclidean space, to infinite dimensions, including Hilbert spaces and separable Banach spaces. Arnold's monograph \cite{Arnold-RDS-1998} provides a detailed exposition on the Multiplicative Ergodic Theorem and invariant manifold theory for finite-dimensional RDS. Smooth systems are discussed in \cite{Liu_Qian-book-1995}. For results on infinite-dimensional RDS, refer to \cite{Ruelle-AM-1982, Mohammed_Zhang_Zhao-book-2008, Lian_Lu-book-2010, Caraballo_Duan_Lu_Schmalfuss-ANS-2010, Lu_Schmalfuss-proc-2012, Barreira_Valls-SD-2018} and the references therein.

Center manifolds for finite dimensional RDS have also garnered attention. Wanner \cite{Wanner-DR-1995} discusses invariant manifolds, including center manifolds, in terms of linearization in $\mathbb{R}^n$. Boxler \cite{Boxler-PTRF-1989} proved the existence of center manifolds in this scenario for discrete random maps (random diffeomorphisms). Existence, smooth conjugacy theorems, and Takens-type theorems based on Lyapunov exponents were established by Li and Lu in \cite{Li_Lu-DCDS-B-2016} and by Guo and Shen in \cite{Guo_Shen-RMJM-2016}, in the presence of zero Lyapunov exponents.

Infinite-dimensional RDS hold significant interest not only due to their inherent mathematical richness but also for their applications in understanding stochastic and partial differential equations. Assuming an exponential trichotomy, Chen, Roberts and Duan~\cite{Chen_Roberts_Duan-JDEA-2015} proved the existence and smoothness of center manifolds for a class of stochastic evolution equations with linearly multiplicative noise. In~\cite{Chen_Roberts_Duan-DS-2019}, Chen, Roberts and Duan established the existence of center manifolds for both discrete and continuous infinite-dimensional RDS, assuming an exponential trichotomy, by employing the Lyapunov-Perron method. They provided examples illustrating the application of these results to stochastic evolution equations through their conversion into infinite-dimensional RDS. In a similar vein, Kuehn and Neam\c{t}u \cite{Kuhn_Neamtu-EJP-2023} address the issue of center manifolds for rough partial differential equations, which also translate into center manifolds within the RDS framework.

Exponential trichotomies have played an important role in invariant manifold theory for infinite-dimensional dynamical systems and non-autonomous systems, whether in deterministic or random scenarios, as discussed. In this work, we extend the results on the existence of center manifolds for infinite-dimensional RDS by assuming a generalized trichotomy. 
This type of general assumption was considered in \cite{Bento_Vilarinho-JDDE-2021} for dichotomies, and in this work, it is extended to include a central direction. This generalization allows for various types of behavior beyond exponential along the three subspaces that partition our phase space. In our context, each subspace is governed by a very general type of rate for controlling the growth of the evolution operator, described in terms of a cocycle. In specific cases, these subspaces align with the usual central, stable, and unstable subspaces, but our assumptions are sufficiently general to encompass behaviors beyond exponential-type, such as those found in (non)uniformly (pseudo-)hyperbolic cases. 

This paper is organized as follows. In Section~\ref{se:Notation and preliminaries}, we introduce the notation and provide preliminary content on general random trichotomies, as well as describe auxiliary spaces of functions. These spaces are crucial for managing the nonlinear components of the RDS and for deriving the center manifold as the graph of a suitable regular function. Section~\ref{sec:cont_time} presents the main result for continuous-time RDS (Theorem~\ref{thm:global}), while Section~\ref{sec:disc_time} focuses on proving the discrete-time counterpart of this result (Theorem~\ref{thm:global:disc}). In Section~\ref{section:examples:cont}, we explore continuous-time examples, including tempered exponential trichotomies, and introduce a general framework called $\psi$-trichotomies, which extends beyond exponential bounds. Corresponding discrete-time examples are provided in Section~\ref{section:examples:disc}.

\section{Generalized trichotomies for RDS}\label{se:Notation and preliminaries}

\subsection{Random Dynamical Systems}

Consider \emph{time} $\T=\Z$ or $\T=\R$, and set $\T^- = \T\, \cap\, ]-\infty,0]$ and $\T^+ = \T \cap [0,+\infty[$. 
A \emph{measure-preserving dynamical system} is a quadruplet $\Sigma\equiv(\W, \mathcal{F}, \mathbb{P}, \theta)$, where $(\W, \mathcal{F}, \mathbb{P})$ is a measure space and
\begin{enumerate}[\lb= ,\lm=10mm]
   \item $\theta \colon \T \times \W \to \W$ is measurable;
   \item $\theta^t(\cdot)= \theta(t,\cdot)\colon \W \to \W$ preserves $\mathbb{P}$ for all $t \in \T$;
   \item $\theta^0 = \Id_{\W}$;
   \item $\theta^{t+s}=\theta^t\circ\theta^s$ for
       all $t,s \in\T$.
\end{enumerate}

A (Bochner) measurable \emph{random dynamical system}, henceforth abbreviated as RDS, on a Banach space $X$ over a measure-preserving dynamical system $\Sigma$ with time $\T$ is a map
\begin{equation*}
  \Phi:\T\times\W\times X \to X
\end{equation*}
such that
\begin{enumerate}[\lb=$\roman*)$,\lm=10mm]
   \item $ \Phi(\cdot,\cdot,x)$ is (Bochner) measurable for all $x \in X$;
   \item $\Phi_\w^t(\cdot)=\Phi(t,\w,\cdot) \colon X \to X$ satisfies
      \begin{enumerate}[\lb=$\alph*)$,\lm=6mm]
         \item $\Phi_\w^0=\Id_X$ for all $\w\in\W$;
         \item $\Phi_\w^{t+s}=\Phi_{\theta^s\w}^t\circ\Phi_\w^s$, for all $\w\in\W$ and all $s, t\in\T$.
      \end{enumerate}
\end{enumerate}
When $\Phi_\w^t$ is a bounded linear operator for all $(t,\w)\in\T\x\W$ , the RDS $\Phi$ is called \emph{linear}. 

We may restrict the \textit{driving system} $\Sigma$ to a $\theta^t$-invariant subset $\W'\subset\W$ with $\mathbb{P}$-full measure, getting a  (Bochner) measurable RDS $\Phi\vert_{\T\times\W'\times X}$  over $\Sigma'\equiv(\W',\mathcal{F}',\mathbb{P}\vert_{\mathcal{F}'},\theta\vert_{\W'})$, where $\mathcal{F}'=\{B\cap\W'\colon B\in\mathcal{F}\}$.

\subsection{Generalized Trichotomies}

For every $i\in\set{\cc,\sss,\uu}$, consider a map $P^i \colon \W \x X \to X$,  and set  $P_\w^i(\cdot)=P(\w,\cdot) \colon X \to X$. Let $\cP=(P^\cc,P^\sss,P^\uu)$. A (Bochner) measurable linear RDS $\Phi$ over $\Sigma$ admits a (Bochner) measurable  \textit{$\cP$-invariant splitting} if
\begin{enumerate}[\lb=$\roman*)$,\lm=10mm]
   \item $P^i(\cdot,x)$ is (Bochner) measurable, for all $x \in X$ and every $i\in\set{\cc,\sss,\uu}$;
	\item $P_\w^i$ is a bounded linear projection, for all $\w \in \W$ and every $i\in\set{\cc,\sss,\uu}$;
	\item $P_\w^\cc+P_\w^\sss+P_\w^\uu=\Id$, for all $\w\in\W$;
   \item $P_\w^\cc P_\w^\sss = 0$, for all $\w\in\W$;
   \item $P_{\theta^t\w}^i \Phi_{\w}^t = \Phi_{\w}^t P_\w^i$, for all $(t ,\w) \in \T\x\W$ and  every $i\in\set{\cc,\sss,\uu}$;
\end{enumerate}
Notice that for all $\w\in\W$ and $i,j\in\set{\cc,\sss,\uu}$, with $i \ne j$, we have
$P_\w^iP_\w^j=0$. To shorten the writing during future computations, for $t\in\T$, $\w\in\W$ and $i\in\set{\cc,\sss,\uu}$ we will adopt the notation
	\[
      \Phi_\w^{i,t}=\Phi_\w^tP_\w^i.
	\]

We define the linear subspaces $E_\w^i =P_\w^i (X)$ for each $i\in\set{\cc,\sss,\uu}$. As usual, we identify $E_\w^\cc \times E_\w^\sss \times E_\w^\uu$ and $E_\w^\cc \oplus E_\w^\sss \oplus E_\w^\uu$.
Given the maps
\begin{align*}
    & \alpha^\cc \colon \T \times \W \to (0,+\infty)\\
    & \alpha^\sss \colon \T^+ \times \W \to (0,+\infty)\\
    & \alpha^\uu \colon \T^- \times \W \to (0,+\infty)
\end{align*}
we define $\alpha = (\alpha^\cc, \alpha^\sss, \alpha^\uu)$. Letting $\alpha^i(t,\w)$ be denoted by $\alpha^i_{t,\w}$, we say that a (Bochner) measurable linear RDS $\Phi$ over $\Sigma$ exhibits a \textit{generalized trichotomy with bounds $\alpha$} (or simply an \textit{$\alpha$-trichotomy}) if it admits a (Bochner) measurable $\cP$-invariant splitting satisfying
\begin{enumerate}[\lb=$($T$\arabic*)$,\lm=13mm]
   \item \label{eq:T1} $\|\Phi_\w^{\cc,t}\|\le\alpha^\cc_{t,\w}$  for all $(t,\w) \in \T \times \W$;
   \item \label{eq:T2} $\|\Phi_\w^{\sss,t}\|\le\alpha^\sss_{t,\w}$ for all $(t,\w) \in \T^+ \times \W$;
   \item \label{eq:T3} $\|\Phi_\w^{\uu,t}\|\le\alpha^\uu_{t,\w}$ for all $(t,\w) \in \T^- \times \W$.
\end{enumerate}

In Section~\ref{section:examples:cont} and Section~\ref{section:examples:disc}, we present several examples of generalized trichotomies with both exponential and non-exponential bounds $\alpha$.

\medskip
In the remainder of this article, $\Phi$ will always denote a measurable (when $\mathbb{T} = \mathbb{Z}$) or Bochner measurable (when $\mathbb{T} = \mathbb{R}$) linear RDS on a Banach space $X$ over a measure-preserving dynamical system $\Sigma \equiv (\Omega, \mathcal{F}, \mathbb{P}, \theta)$ exhibiting a trichotomy with bounds $\alpha = (\alpha^\cc, \alpha^\sss, \alpha^\uu)$.

\subsection{Auxiliary spaces}
Let $\sF$ denote the space of maps $f \colon \W \times X \to X$ for which $f(\cdot, x)$ is measurable for every $x \in X$, and for which, setting $f_\w(\cdot) = f(\w, \cdot)$, for every $\w \in \W$ we have
\begin{equation}
   \label{eq:f_w(0)=0}
      f_\w(0) = 0
\end{equation}
and
\begin{equation}
    \label{eq:Lip(f_w)}
      \Lip(f_\w) = \sup \left\{ \frac{\|f_\w(x) - f_\w(y)\|}{\|x - y\|} \colon x, y \in X, \ x \ne y \right\} < +\infty.
\end{equation}
Conditions~\eqref{eq:Lip(f_w)} and~\eqref{eq:f_w(0)=0} ensure that for all $\w \in \W$ and for all $x, y \in X$
\begin{equation}
   \label{ine:||f_w(x)-f_w(y)||<=Lip(f_w)||x-y||}
   \|f_\w(x) - f_\w(y)\| \le \Lip(f_\w) \|x - y\|,
\end{equation}
and
\begin{equation}    
    \label{ine:||f_w(x)||<=Lip(f_w)||x||}
    \|f_\w(x)\| \le \Lip(f_\w) \|x\|.
\end{equation}
Let $\sF^{(B)}$ represent the collection of functions $f \in \sF$ for which $f(\cdot, x)$ is Bochner measurable for each $x \in X$. Additionally, define $\sF^{(B)}_{\alpha}$ as the subset of $\sF^{(B)}$ consisting of functions $f$ such that, for every $\w \in \W$, the functions   \begin{align*}
 &
   [a,b] \ni r \mapsto \alpha^\cc_{t-r,\theta^r \w} \Lip(f_{\theta^r \w}) \alpha^\cc_{r,\w}
\\
 & [c,0] \ni  r \mapsto \alpha^\sss_{-r,\theta^r \w} \Lip(f_{\theta^r \w}) \alpha^\cc_{r,\w}
\\
&   [0,d] \ni r \mapsto \alpha^\uu_{-r,\theta^r \w} \Lip(f_{\theta^r \w}) \alpha^\cc_{r,\w}
\end{align*}
are measurable for every $a< b$, $c<0$, $d>0$ and $t\in\R$.
 
\medskip
We define the set
\[
\cH = \{(t, \w, \xi) \in \T \times \W \times X \colon \xi \in E_{\w}^\cc\}.
\]
For a given $M > 0$, let $\fJ_M$ (resp. $\fJ^{(B)}_M$) denote the space of all functions $h \colon \cH \to X$ such that, for each $(t, \w) \in \T \times \W$, the map $h_{t, \w}(\cdot) = h(t, \w, \cdot)$ satisfies
\begin{align}
   & \label{eq:(t,w)|->h_t,w(xi):Borel}
      h(\cdot,\cdot,P_\w^\cc x )
      \text{ is measurable (resp. Bochner measurable) for all } x \in X;\\
   & \label{eq:h_n,w(0)=0}
      h_{t,\w}(0) =0
      \text{ for all } (t,\w)\in\T\x\W;\\
   & \label{eq:h_0,w(xi)=xi}
      h_{0,\w}=\Id_{E_\w^\cc}
      \text{ for all } \w \in \W;\\
   & \label{h_n,w(xi):in:E(t^n(w))}
      h_{t,\w}(E_{\w}^\cc) \subseteq E_{\theta^t\w}^\cc
      \text{ for all } (t,\w) \in \T\x\W;\\
   & \label{ine:||h_n,w(xi)-h_n,w(bxi)||<=M...}
      \|h_{t,\w}(\xi)-h_{t,\w}(\xi')\| \le M \alpha^\cc_{t,\w}\|\xi-\xi'\| 
      \text{ for all } (t,\w,\xi), (t,\w,\xi') \in \cH.
\end{align}
From~\eqref{ine:||h_n,w(xi)-h_n,w(bxi)||<=M...} and~\eqref{eq:h_n,w(0)=0} it follows that
\begin{equation}\label{ine:||h_n,w(xi)||<=M...}
   \|h_{t,\w}(\xi)\| \le M \alpha^\cc_{t,\w}\|\xi\|\,
      \text{ for all } (t,\w,\xi) \in \cH.
\end{equation}
Defining
\begin{equation} \label{def:d_1(x,y)}
   d_1(h,g)
   = \sup\set{\dfrac{\|h_{t,\w}(\xi)-g_{t,\w}(\xi)\|}{\alpha^\cc_{t,\w} \|\xi\|} \colon
      (t,\w,\xi) \in \cH, \ \xi \neq 0}
\end{equation}
we have that  $({\fJ_M},d_1)$ and $(\fJ_M^{(B)},d_1)$ are complete metric spaces. 

\medskip
We now consider the set
\[
\cG = \{(\w, \xi) \in \W \times X \colon \xi \in E_{\w}^\cc\}.
\]
For a given $N > 0$, let $\cL_N$ (resp. $\cL_N^{(B)}$) denote the space of all functions $\phi \colon \cG \to X$ such that, for each $\w \in \W$, the map $\phi_\w(\cdot) = \phi(\w, \cdot)$ satisfies
\begin{align}
   & \label{eq:w|->phi_w(xi):Borel}
       \phi(\cdot,P_\w^\cc x)
      \text{ is measurable (resp. Bochner measurable)  for all } x \in X;\\
   &
      \phi_\w(0)=0 \text{ for all } \w \in \W; \label{eq:phi_w(0)=0} \\
   & \label{eq:phi_w(xi):in:F_w}
      \phi_\w(E_\w^\cc) \subseteq E_\w^\sss\oplus E_\w^\uu \text{ for all } \w\in\W; \\
   & \label{eq:||phi_w(xi)-pi_w(bxi)||<=...}
      \|\phi_\w(\xi) - \phi_\w(\xi')\| \le N \|\xi - \xi'\|
      \text{ for all } (\w,\xi), (\w, \xi') \in \cG.
\end{align}
By~\eqref{eq:||phi_w(xi)-pi_w(bxi)||<=...} and~\eqref{eq:phi_w(0)=0}, taking $\xi' = 0$ we get
\begin{equation} \label{ine:||phi_w(xi)||<=N||xi||}
   \|\phi_\w(\xi)\| \le N \|\xi\|
   \text{ for all  } (\w, \xi) \in \cG.
\end{equation}
For future use, we set the notation $\phi^\sss_\w = P^\sss_\w \phi_\w$ and $\phi^\uu_\w = P^\uu_\w \phi_\w$. 
Given $\phi \in \cL_N$ and $\w \in \W$ we denote the \emph{graph} of $\phi_\w$ by
\begin{equation*}
   \cV_{\phi,\w}
   = \set{(\xi,\phi_\w(\xi)): \xi \in E_\w^\cc}\subseteq X.
\end{equation*} 
Defining now
\begin{equation}\label{def:d_2(phi,psi)}
   d_2(\phi,\psi)
   = \sups{\dfrac{\|\phi_\w(\xi) - \psi_\w(\xi)\|}{\|\xi\|}
      \colon (\w,\xi)\in\cG,  \xi\neq 0}
\end{equation}
we have that $(\cL_N,d_2)$ and
$(\cL_N^{(B)},d_2)$ are complete metric spaces. 

\medskip
To finalize this section, let $\cX_{M,N}=\fJ_M \times
\cL_N$ and  $\cX^{(B)}_{M,N}=\fJ^{(B)}_M \x \cL^{(B)}_N$. Setting
$$ d\prts{(h,\phi),(g,\psi)} = d_1(h,g) + d_2(\phi,\psi),$$
we also have that $(\cX_{M,N},d)$ and $(\cX^{(B)}_{M,N},d)$ 
are complete metric spaces.

\section{Invariant manifolds in continuous-time RDS} \label{sec:cont_time}
Throughout this section, we focus on the continuous-time case by considering \(\mathbb{T} = \mathbb{R}\).
Given a Bochner measurable linear RDS $\Phi$ and a map $f \in \sF^{(B)}_\alpha$, we define
\begin{equation}\label{def:sigma}
   \sgm
   = \sup\limits_{(t,\w) \in \R \x \W} \
      \dfrac{1}{\alpha^\cc_{t,\w}}
      \abs{
      \dint_{0}^{t} \alpha^\cc_{t-r,\theta^r\w} \Lip(f_{\theta^r\w}) \alpha^\cc_{r,\w}\dr}
\end{equation}
and
\begin{equation}\label{def:tau}
   \tau
   = \sup_{\w\in\W} \dint_{-\infty}^{0} \
      \alpha^\sss_{-r,\theta^r\w} \Lip(f_{\theta^r\w}) \alpha^\cc_{r,\w}\dr
   +
      \dint_{0}^{+\infty} \
      \alpha^\uu_{-r,\theta^r\w} \Lip(f_{\theta^r\w}) \alpha^\cc_{r,\w}\dr.
\end{equation}
If for every $(\w,x)\in\W\x X $ there is a unique solution $\Psi(\cdot,\w,x)$ of the equation
\begin{equation}\label{eq:u(t)=...}
   u(t) = \Phi_\w^t x
      + \dint_0^t \Phi_{\theta^{r}\w}^{t-r} f_{\theta^r \w}(u(r))\,dr
\end{equation}
then 
$\Psi\colon \R \x \W \x X \to X$ is a Bochner measurable RDS on $X$ over $\Sigma$. In particular, 
$
\Psi(\cdot,\cdot,x)$ is Bochner measurable for all $x\in X$, and
\begin{equation}\label{eq:Psi:cont}
   \Psi_\w^t x
   = \Phi_\w^t x
      + \dint_0^t \Phi_{\theta^{r}\w}^{t-r} f_{\theta^r \w}(\Psi_\w^r x)\,dr.
\end{equation}

\begin{theorem} \label{thm:global}
   Let $\Phi$ be a Bochner measurable linear RDS exhibiting an $\alpha$-trichotomy, and let $f \in \sF^{(B)}_\alpha$. Suppose that $\Psi$ is a Bochner measurable RDS such that$\Psi(\cdot,\w,x)$ is the unique solution of~\eqref{eq:u(t)=...} for all $(\w,x) \in \W\x X$. If
   \begin{equation} \label{eq:CondicaoTeo}
      \lim_{t \to -\infty} \alpha^\sss_{-t,\theta^t\w} \alpha^\cc_{t,\w}
      = \lim_{t \to +\infty}  \alpha^\uu_{-t,\theta^t\w} \alpha^\cc_{t,\w}
      = 0
   \end{equation}
   for all $\w\in\W$, and
   \begin{equation}\label{eq:CondicaoTeo:alpha+tau<1/2}
      \sgm +  \tau < 1/2,
   \end{equation}
   then there are $N \in\ ]0,1[$ and a unique $\phi \in \cL^{(B)}_N$ such that
   \begin{equation} \label{thm:global:invar}
      \Psi_{\w}^t(\cV_{\phi,\w})
      \subseteq \cV_{\phi,\theta^t\w}
   \end{equation}
   for all $(t,\w) \in \R \x \W$.   Moreover,   for  all $(t,\w, \xi), (t,\w, \xi') \in \cH$ we have
   \begin{equation}\label{thm:ineq:norm:F_mn(xi...)-F_mn(barxi...)}
      \|\Psi_{\w}^t(\xi,\phi_\w(\xi)) - \Psi_{\w}^t(\xi',\phi_\w(\xi'))\|
      \le ({N}/{\tau}) \alpha^\cc_{t,\w} \, \|\xi - \xi'\|.
   \end{equation}
\end{theorem}

The remaining part of this section is devoted to prove Theorem~\ref{thm:global}. 

\medskip
From \cite[Lemma 5.1]{Bento_Costa-EJQTDE-2017},  we may found constants $M \in\ ]1,2[$ and $N \in\ ]0,1[$ such that
   \begin{equation}\label{def:M+N}
      \sgm = \dfrac{M-1}{M(1+N)}
      \ \ \ \text{ and } \ \ \
      \tau = \dfrac{N}{M(1+N)}.
   \end{equation}

\begin{lemma}
   Consider $(h,\phi) \in \cX^{(B)}_{M,N}$. 
   \begin{enumerate}[\lb=$\alph*)$,\lm=6mm]
      \item For every $x \in X$ the maps
         \begin{equation*}
            \begin{split}
               &  (t,r,\w) \mapsto \Phi^{\cc,t-r}_{\theta^r \w}
                  f_{\theta^r \w} \prts{h_{r,\w}(P^\cc_\w x),
                     \phi_{\theta^r \w}\prts{h_{r,\w}\prts{P^\cc_\w x}}}\\
               & (r,\w) \mapsto \Phi^{\sss,-r}_{\theta^r \w}
                  f_{\theta^r \w} \prts{h_{r,\w}(P^\cc_\w x),
                     \phi_{\theta^r \w}\prts{h_{r,\w}\prts{P^\cc_\w x}}}\\
               & (r,\w) \mapsto \Phi^{\uu,-r}_{\theta^r \w}
                  f_{\theta^r \w} \prts{h_{r,\w}(P^\cc_\w x),
                     \phi_{\theta^r \w}\prts{h_{r,\w}\prts{P^\cc_\w x}}}
            \end{split}
         \end{equation*}
         are Bochner measurable on $\R\times\R\times\W$, $\R^-\times\W$ and $\R^+\times\W$, respectively.
      \item For every $(t,\w,x) \in \R \x \W \x X$ the map
         \begin{equation*}
            r \mapsto \Phi^{\cc,t-r}_{\theta^r \w}
               f_{\theta^r \w} \prts{h_{r,\w}(P^\cc_\w x),
                  \phi_{\theta^r \w}\prts{h_{r,\w}\prts{P^\cc_\w x}}}
         \end{equation*}
         is Bochner integrable in every closed interval with bounds $0$ and $t$.
      \item For every $(\w,x) \in \W \x X$ and $t>0$, the maps
      \begin{equation*}
\begin{split}            
            &r \mapsto \Phi^{\sss,-r}_{\theta^r \w}
               f_{\theta^r \w} \prts{h_{r,\w}(P^\cc_\w x),
                  \phi_{\theta^r \w}\prts{h_{r,\w}\prts{P^\cc_\w x}}}
\\
&
            r \mapsto \Phi^{\uu,-r}_{\theta^r \w}
               f_{\theta^r \w} \prts{h_{r,\w}(P^\cc_\w x),
                  \phi_{\theta^r \w}\prts{h_{r,\w}\prts{P^\cc_\w x}}}
\end{split} 
         \end{equation*}
         are Bochner integrable in $[-t,0]$ and $[0,t]$, respectively.
   \end{enumerate}
   \end{lemma}
The proof follows similarly as \cite[Lemma 3.6]{Bento_Vilarinho-JDDE-2021}.
Given $\w\in \W$ and $x_\w=(x^\cc_\w, x^\sss_\w, x^\uu_\w)\in E^\cc_\w \times E^\sss_\w \times E^\uu_\w$,
it follows from~\eqref{eq:Psi:cont} that the trajectory
$x_{\theta^t\w} = \Psi^t_\w x_\w=\prts{x^\cc_{\theta^t\w}, x^\sss_{\theta^t\w}, x^\uu_{\theta^t\w}}$ satisfies, for all $i\in\{\cc,\sss,\uu\}$
and all $t \in \R$,
\begin{equation}\label{eq:dyn-split1x}
      x^i_{\theta^t\w}
      = \Phi^{i,t}_{\w} x_\w
         + \int_0^t \Phi^{i,t-s}_{\theta^s\w} f_{\theta^s\w}
         \prts{x^\cc_{\theta^s\w}, x^\sss_{\theta^s\w}, x^\uu_{\theta^s\w}}\ds.
\end{equation}

Taking into account the invariance required in~\eqref{thm:global:invar}, for any given $x_\w\in\cV_{\phi,\w}$ and $t \in \R$ we must have $x_{\theta^t\w}\in\cV_{\phi,\theta^t\w}$. Thus, in this situation, the equations given by~\eqref{eq:dyn-split1x} can be written as
\begin{align*}
   &
      x^\cc_{\theta^t\w}
      = \Phi_{\w}^{\cc,t} x_\w + \int_0^t \Phi_{\theta^s\w}^{\cc,t-s}
         f_{\theta^s\w} \prts{x^\cc_{\theta^s\w},\phi_{\theta^s\w}\prts{x^\cc_{\theta^s\w}}}\ds,
   \\
   &
      \phi^\sss_{\theta^t\w}\prts{x_{\theta^t\w}}
      = \Phi_{\w}^{t}  \phi^\sss_{\w}\prts{x^\cc_\w} + \int_0^t
         \Phi_{\theta^s\w}^{\sss,t-s} f_{\theta^s\w} \prts{x^\cc_{\theta^s\w},\phi_{\theta^s\w}\prts{x^\cc_{\theta^s\w}}}\ds,
   \\
   &
      \phi^\uu_{\theta^t\w}\prts{x_{\theta^t\w}}
      = \Phi_{\w}^{t}  \phi^\uu_{\w}\prts{x^\cc_\w} + \int_0^t
         \Phi_{\theta^s\w}^{\uu,t-s} f_{\theta^s\w} \prts{x^\cc_{\theta^s\w},\phi_{\theta^s\w}\prts{x^\cc_{\theta^s\w}}}\ds.
\end{align*}

\begin{lemma} \label{lemma:equiv}
   Consider $(h,\phi) \in \cX^{(B)}_{M,N}$ such that, for all $(t,\w,\xi)\in \cH$,
   \begin{equation} \label{eq:dyn-split3a}
      h_{t,\w}(x)
      = \Phi_{\w}^t \xi + \int_0^t \Phi_{\theta^r\w}^{\cc,t-r}
         \fthr \dr.
   \end{equation}
   The following properties $a)$ and $b)$ are equivalent:
   \begin{enumerate}[\lb=$\alph*)$,\lm=5mm]
      \item For each $j\in\{\sss,\uu\}$ and all $(t,\w,\xi)\in \cH$,
         \begin{equation}\label{eq:dyn-split3x}
             \phi^j_{\theta^t\w}\prts{h_{t,\w}\prts{\xi}}
             = \Phi_{\w}^t  \phi^j_{\w}(\xi) 
               + \int_0^t \Phi_{\theta^r\w}^{j,t-r} \fthr \dr
         \end{equation}
      \item For all $(\w,\xi) \in \cG$
        \begin{equation} \label{eq:phi^+_w(xi)=...}
            \phi^\sss_{\w}(\xi)
               = \int_{-\infty}^0 \Phi^{\sss,-r}_{\theta^r \w} \fthr \dr
        \end{equation}
        and
        \begin{equation} \label{eq:phi^-_w(xi)=...}
            \phi^\uu_{\w}(\xi)
               = - \int_0^{+\infty} \Phi^{\uu,-r}_{\theta^r \w} \fthr \dr.
        \end{equation}
   \end{enumerate}
\end{lemma}

\begin{proof}
 From~\eqref{ine:||f_w(x)||<=Lip(f_w)||x||},~\eqref{ine:||phi_w(xi)||<=N||xi||} and~\eqref{ine:||h_n,w(xi)||<=M...} we have
   \begin{align*}
       \norm{f_{\theta^r \w}(h_{r,\w}(\xi), \phi_{\theta^r \w}(h_{r,\w}(\xi)))}
      & \le \Lip(f_{\theta^r \w}) \prts{\norm{h_{r,\w}(\xi)}
         + \norm{\phi_{\theta^r \w}(h_{r,\w}(\xi))}}
         \\
      & \le M(1+N) \Lip(f_{\theta^r \w}) \alpha^{\cc}_{r,\w} \|\xi\|
   \end{align*}
for every $(\w,\xi) \in \cG$.
Thus, by~\ref{eq:T2},
\begin{equation*}
  \int_{-\infty}^0 \norm{\Phi_{\theta^r\w}^{\sss,-r} \fthr} \dr
         \le M(1+N) \tau \|\xi\|,
   \end{equation*}
   and by~\ref{eq:T3} we obtain
   \begin{equation*}
\int_0^{+\infty} \norm{\Phi_{\theta^r\w}^{\uu,-r} \fthr} \dr \le M(1+N) \tau \|\xi\|.
   \end{equation*}
Hence the integrals are convergent.  

Suppose that~\eqref{eq:dyn-split3x} holds for $j=\sss$ and all $(t,\w,\xi) \in \cH$. By applying $\Phi_{\theta^t \w}^{-t}$ to both sides, it is equivalent to
   \begin{equation} \label{eq:equiv1A}
      \phi^\sss_{\w}(\xi)
      = \Phi^{\sss,-t}_{\theta^t \w} \phi^\sss_{\theta^t\w}(h_{t,\w}(\xi))
         - \int_0^t \Phi^{\sss,-r}_{\theta^r \w} \fthr \dr.
   \end{equation}\normalsize
   Using~\ref{eq:T2},~\eqref{ine:||phi_w(xi)||<=N||xi||} and~\eqref{ine:||h_n,w(xi)||<=M...}, for $t \le 0$
   we have
   \begin{equation*}
         \norm{\Phi^{\sss,-t}_{\theta^t \w} \phi^\sss_{\theta^t\w}(h_{t,\w}(\xi))}
      \le MN \alpha^\sss_{-t,\theta^t \w} \ \alpha^\cc_{t,\w} \|\xi\|,
   \end{equation*}
  which converges to zero as $t \to -\infty$ by~\eqref{eq:CondicaoTeo}. Thus, by taking $t \to -\infty$   in equation~\eqref{eq:equiv1A}  we obtain~\eqref{eq:phi^+_w(xi)=...}. 
   Similarly, equation~\eqref{eq:dyn-split3x} with $j=\uu$ can be written as
   \begin{equation} \label{eq:equiv1B}
      \phi^\uu_{\w}(\xi)
      = \Phi^{\uu,-t}_{\theta^t \w} \phi_{\theta^t\w}(h_{t,\w}(\xi))
         - \int_0^t \Phi^{\uu,-r}_{\theta^r \w} \fthr \dr.
   \end{equation}
   Using~\ref{eq:T3},~\eqref{ine:||phi_w(xi)||<=N||xi||} and~\eqref{ine:||h_n,w(xi)||<=M...}, for $t \ge 0$ we have
   \begin{equation*}
      \norm{\Phi^{\uu,-t}_{\theta^t \w} \phi_{\theta^t\w}(h_{t,\w}(\xi))}
      \le MN \alpha^\uu_{-t,\theta^t\w} \alpha^\cc_{t,\w} \norm{\xi},
   \end{equation*}
   which by~\eqref{eq:CondicaoTeo} converges to zero as $t \to +\infty$. Thus we obtain~\eqref{eq:phi^-_w(xi)=...} by taking $t \to +\infty$ in equation~\eqref{eq:equiv1B}.

   For the converse, assume now that~\eqref{eq:phi^+_w(xi)=...} and~\eqref{eq:phi^-_w(xi)=...} hold for all $(\w,\xi) \in \cG$. For all $t \in\R$ , we have 
   \begin{align*}
         \Phi_\w^t \phi^\sss_\w(\xi)
      &
         = \int_t^0  \Phi^{\sss,t-r}_{\theta^r \w} \fthr \dr\\
      &
          \qquad + \int_{-\infty}^0  \Phi^{\sss,-r}_{\theta^{t+r} \w}
            f_{\theta^{t+r} \w} \prts{h_{t+r,\w}\prts{\xi},
               \phi_{\theta^{t+r} \w}\prts{h_{t+r,\w}\prts{\xi}}}\dr
   \end{align*}
 and
   \begin{align*}
         \Phi_\w^t \phi^\uu_\w(\xi)
      &
         = - \int_0^t  \Phi^{\uu,t-r}_{\theta^r \w} \fthr \dr\\
      &
          \qquad - \int_{-\infty}^0  \Phi^{\uu,-r}_{\theta^{t+r} \w}
            f_{\theta^{t+r} \w} \prts{h_{t+r,\w}\prts{\xi},
               \phi_{\theta^{t+r} \w}\prts{h_{t+r,\w}\prts{\xi}}}\dr.
   \end{align*}
  Since $h_{t+s,\w}(\xi)=h_{s,\theta^t\w}(h_{t,\w}(\xi))$ due to the uniqueness of the solution of~\eqref{eq:u(t)=...}, we get the identity~\eqref{eq:dyn-split3x} for $j=\sss$ and $j=\uu$. 
\end{proof}

Consider the operator \( \JJ \), which assigns each pair \( (h, \phi) \in \cX^{(B)}_{M,N} \) to the map \( \JJ(h, \phi) \colon \cH \to X \) given by
\begin{align*}
   \prtsr{\JJ\prts{h,\phi}}(t, \omega, \xi)
   & = \Phi_\omega^t \xi + \int_0^t \Phi_{\theta^r \omega}^{\cc, t-r} \phi(r, \omega) \, dr.
\end{align*}

\begin{lemma}
 $\JJ\prts{\cX^{(B)}_{M,N}}\subseteq\fJ^{(B)}_M.$
\end{lemma}

\begin{proof}
   Fix a pair $(h,\phi) \in \cX^{(B)}_{M,N}$ . It is straightforward to check that $\JJ\prts{h,\phi}$ satisfies conditions
   \eqref{eq:(t,w)|->h_t,w(xi):Borel}~to~\eqref{h_n,w(xi):in:E(t^n(w))}.        Define
      $$ \gamma_{\theta^r\w}(\xi,\xi')
         = \| f_{\theta^r\w}(h_{r,\w}(\xi),\phi_{\theta^r\w}(h_{r,\w}(\xi)))
            - f_{\theta^r\w}(h_{r,\w}(\xi'),\phi_{\theta^r\w}(h_{r,\w}(\xi')))\|.$$
   From~\eqref{ine:||f_w(x)-f_w(y)||<=Lip(f_w)||x-y||},~\eqref{eq:||phi_w(xi)-pi_w(bxi)||<=...} and~\eqref{ine:||h_n,w(xi)-h_n,w(bxi)||<=M...} we have
   \begin{equation}\label{eq:||f_k(x_k,n(xi),...)-f_k(x_k,n(barxi),...)||}
      \gamma_{\theta^r\w}(\xi,\xi')
      \le \Lip(f_{\theta^r\w}) M(1+N) \|\xi-\xi'\| \alpha^\cc_{r,\w}.
   \end{equation}  
   Following the previous notation,  $\JJ\prts{h,\phi}_{t,\w}(\xi)$ stands for $[\JJ\prts{h,\phi}](t,\w,\xi)$.    By~\ref{eq:T1}, \eqref{def:sigma}, \eqref{eq:||f_k(x_k,n(xi),...)-f_k(x_k,n(barxi),...)||} and~\eqref{def:M+N}, we have
   \begin{equation*}
       \begin{split}
           \| \JJ\prts{h,\phi}_{t,\w}(\xi)
            - \JJ\prts{h,\phi}_{t,\w}(\xi')\|
         & \le \|\Phi^{\cc,t}_\w \| \|\xi - \xi'\|
            + \dint_0^t \|\Phi_{\theta^r\w}^{\cc,t-r}\|  \gamma_{\theta^r\w}(\xi,\xi')\dr\\
         & \leq \prts{1 + \sgm M(1+N) } \alpha^\cc_{t,\w} \|\xi - \xi'\|\\
         & = M  \alpha^\cc_{t,\w}\|\xi - \xi'\|.
      \end{split}
   \end{equation*}
   Hence $\JJ\prts{h,\phi}$ also satisfies
   \eqref{ine:||h_n,w(xi)-h_n,w(bxi)||<=M...}.
\end{proof}

Consider now the operator $\LL$, which assigns each pair $(h,\phi)\in\cX^{(B)}_{M,N}$ the map $\LL(h,\phi)\colon \cG \to X$ given by
\begin{equation*}
   \prtsr{\LL\prts{h,\phi}}(\w,\xi)
   = \prtsr{\LL^\sss\prts{h,\phi}}(\w,\xi) + \prtsr{\LL^\uu\prts{h,\phi}}(\w,\xi)
\end{equation*}
where
\begin{equation*}
   \prtsr{\LL^\sss \prts{h,\phi}}(\w,\xi)
   = \int_{-\infty}^0 \Phi^{\sss,-r}_{\theta^r \w} \fthr \dr
\end{equation*}
and
\begin{equation*}
   \prtsr{\LL^\uu\prts{h,\phi}}(\w,\xi)
   = - \int_0^{+\infty} \Phi^{\uu,-r}_{\theta^r \w} \fthr \dr.
\end{equation*}

\begin{lemma}
   $\LL\prts{\cX^{(B)}_{M,N}}\subseteq\cL^{(B)}_N.$
\end{lemma}

\begin{proof}
   Fix $\prts{h,\phi}\in\cX^{(B)}_{M,N}$. It is immediate to check that $\prtsr{\LL\prts{h,\phi}}(\w,\xi)$ satisfies conditions \eqref{eq:w|->phi_w(xi):Borel}~to~\eqref{eq:phi_w(xi):in:F_w}.
   Again, $\LL\prts{h,\phi}_\w(\xi)$ stands for  $\prtsr{\LL\prts{h,\phi}}(\w,\xi)$. From~\ref{eq:T2},~\ref{eq:T3}, ~\eqref{eq:||f_k(x_k,n(xi),...)-f_k(x_k,n(barxi),...)||},~\eqref{def:tau} and~\eqref{def:M+N} we have
   \begin{equation*}
      \begin{split}
         & \norm{\LL\prts{h,\phi}_\w(\xi) - \LL\prts{h,\phi}_\w(\xi')}\\
         & \le \dint_{-\infty}^{0} \norm{\Phi^{\sss,-r}_{\theta^r \w}}
            \, \gamma_{\theta^r\w}(\xi,\xi') \dr +\dint_0^{+\infty} \norm{\Phi^{\uu,-r}_{\theta^r \w}}
            \, \gamma_{\theta^r\w}(\xi,\xi') \dr\\
         & \le \tau M \prts{1 + N} \norm{\xi-\xi'}\\
         & = N \norm{\xi-\xi'}.
      \end{split}
   \end{equation*}
   Hence~\eqref{eq:||phi_w(xi)-pi_w(bxi)||<=...} also holds for $\LL(h,\phi)$.
\end{proof}

Consider now $\TT \colon \cX^{(B)}_{M,N}  \to \cX^{(B)}_{M,N} $ given by
   $$ \TT(h,\phi) = \prts{\JJ(h,\phi), \LL(h,\phi)}.$$

\begin{lemma}
   The operator $\TT$ is a contraction in $ (\cX^{(B)}_{M,N},d)$.
\end{lemma}

\begin{proof}
   Consider $\prts{h,\phi}, \prts{g,\psi} \in \cX^{(B)}_{M,N}$. Define
   \begin{equation*}
               \hat\gamma_{\theta^r\w}(\xi)= \|f_{\theta^r\w}(h_{r,\w}(\xi),\phi_{\theta^r\w}(h_{r,\w}(\xi)))
            - f_{\theta^r\w}(g_{r,\w}(\xi),\psi_{\theta^r\w}(g_{r,\w}(\xi)))\|.        
   \end{equation*}
  By~\eqref{ine:||f_w(x)-f_w(y)||<=Lip(f_w)||x-y||}, \eqref{eq:||phi_w(xi)-pi_w(bxi)||<=...}, ~\eqref{def:d_1(x,y)},~\eqref{def:d_2(phi,psi)} and~\eqref{ine:||h_n,w(xi)||<=M...}, for all $(r,\w) \in \R_0^+\x\W$ and  all $\xi\in E_\w$,
   \begin{equation}
      \begin{split}
          \hat\gamma_{\theta^r\w}(\xi)
         & \le \Lip(f_{\theta^r\w})  \alpha^\cc_{r,\w}
            \prts{(1+N)d_1(h,g) + M d_2(\phi,\psi)}\|\xi\|.\label{ine:hat_gamma<=...}
      \end{split}
   \end{equation}
    Hence,  in one hand, from~\ref{eq:T1}, \eqref{ine:hat_gamma<=...} and~\eqref{def:sigma}, we have
   \begin{align*}
         \|\JJ(h,\phi)_{t,\w}(\xi) - \JJ(g,\psi)_{t,\w}(\xi)\|
      & 
         \le  \dint_0^t \norm{\Phi_{\theta^r\w}^{\cc,t-r}} \hat\gamma_{\theta^r\w}(\xi) \dr\\
      &
         \le \sgm  \alpha^\cc_{t,\w}\prts{(1+N)d_1(h,g)
            + M d_2(\phi,\psi)}\norm{\xi},
   \end{align*}
   which implies
      $$ d_1\prts{\JJ(h,\phi),\JJ(g,\psi)}
         \le \sgm   \prts{(1+N)d_1(h,g) + M d_2(\phi,\psi)}.$$
   On the other hand, from~\ref{eq:T2},~\ref{eq:T3},~\eqref{ine:hat_gamma<=...} and~\eqref{def:tau} we get
   \begin{align*}
      & 
         \|\LL\prts{h,\phi}_\w(\xi) - \LL\prts{g,\psi}_\w(\xi)\|\\
      &
         \le \dint_{-\infty}^{0} \norm{\Phi^{\sss,-r}_{\theta^r \w}}
            \, \hat\gamma_{\theta^r\w}(\xi) \dr +\dint_0^{+\infty} \norm{\Phi^{\uu,-r}_{\theta^r \w}}
            \, \hat\gamma_{\theta^r\w}(\xi) \dr\\
      & \le \tau 
         \prts{(1+N)d_1(h,g) + M d_2(\phi,\psi)}\|\xi\|,
   \end{align*}
which implies
      $$ d_2\prts{\LL(h,\phi),\LL(g,\psi)}
         \le \tau \prts{(1+N)d_1(h,g) + M d_2(\phi,\psi)}.$$
In overall we get
   \begin{equation*}
      \begin{split}
         d\prts{\TT(h,\phi),\TT(g,\psi)}
         & \le \prts{\sgm + \tau} \prts{(1+N)d_1(h,g)+ M d_2(\phi,\psi)} \\
         & \le \dfrac{1}{2} \maxs{1+N,M} d((h,\phi),(g,\psi))
      \end{split}
   \end{equation*}
   and because $N<1$ and $M<2$, $\TT$ is a contraction.
\end{proof}

\begin{proof}[Proof of Theorem~\ref{thm:global}] 
   Since $\TT$ is a contraction, by the Banach Fixed Point Theorem, $\TT$ has a unique fixed point $(h,\phi)$, that satisfies~\eqref{eq:dyn-split3a}, \eqref{eq:phi^+_w(xi)=...} and~\eqref{eq:phi^-_w(xi)=...}. By Lemma~\ref{lemma:equiv}, the pair $(h,\phi)$ also satisfies conditions~\eqref{eq:dyn-split3x}. Therefore, for given initial condition $x_\w=(\xi,\phi_\w^\sss(\xi),\phi_\w^\uu(\xi)) \in E_\w^\cc\times E_\w^\sss\times E_\w^\uu$, the trajectory $x_{\theta^t\w}=\prts{h_{t,\w}(\xi), \phi_{\theta^t\w}(h_{t,\w}(\xi))}$ is the solution of~\eqref{eq:u(t)=...}. The graphs $\cV_{\phi,\w}$ are the required invariant manifolds of $\Psi$. To obtain \eqref{thm:ineq:norm:F_mn(xi...)-F_mn(barxi...)}, it follows from~\eqref{eq:||phi_w(xi)-pi_w(bxi)||<=...},~\eqref{ine:||h_n,w(xi)-h_n,w(bxi)||<=M...} and~\eqref{def:M+N} that, for each $(t,\w,\xi), (t,\w,\xi')\in\cH$
   \begin{align*}
      &
         \|\Psi_{\w}^t(\xi,\phi_\w^\sss(\xi),\phi_\w^\uu(\xi)) - \Psi_{\w}^t(\xi',\phi_\w^\sss(\xi'),\phi_\w^\uu(\xi'))\|\\
      & 
         = \|\prts{h_{t,\w}(\xi),\phi_{\theta^t\w}(h_{t,\w}(\xi))}
            - \prts{h_{t,\w}(\xi'),\phi_{\theta^t\w}(h_{t,\w}(\xi'))}\|\\
      & 
         \le M(1+N) \alpha^\cc_{t,\w} \|\xi-\bar \xi\|\\
      & 
         \le\frac{N}{\tau} \alpha^\cc_{t,\w} \|\xi-\bar \xi\|.
   \end{align*}
\end{proof}

\section{Invariant manifolds in discrete-time RDS}\label{sec:disc_time}

Throughout this section we consider $\T = \Z$.
Given a measurable linear RDS $\Phi$  and a map $f \in \sF$, we define
\begin{align*}
   \sgm_\w^-
   &= \sup\limits_{n\in\N}\dfrac{1}{\alpha^\cc_{-n,\w}}
      \dsum_{k=-n}^{-1} \alpha^\cc_{-n-k-1,\theta^{k+1}\w} \Lip(f_{\theta^k\w}) \alpha^\cc_{k,\w}\\
   \sgm_\w^+
   &= \sup\limits_{n\in\N}\dfrac{1}{\alpha^\cc_{n,\w}}
      \dsum_{k=0}^{n-1} \alpha^\cc_{n-k-1,\theta^{k+1}\w} \Lip(f_{\theta^k\w}) \alpha^\cc_{k,\w}
\end{align*}
and 
\begin{equation*}
   \sgm
   = \sup\limits_{\w\in\W}\max\left\{\sgm_\w^-,\sgm_\w^+\right\}.
\end{equation*}
Moreover, writing
\begin{align*}
	\tau_\w^-
	&= \dsum_{k=-\infty}^{-1}
      \alpha^\sss_{-k-1,\theta^{k+1}\w} \Lip(f_{\theta^k\w}) \alpha^\cc_{k,\w}\\
      \tau_\w^+
      &=\dsum_{k=0}^{+\infty}
      \alpha^\uu_{-k-1,\theta^{k+1}\w} \Lip(f_{\theta^k\w}) \alpha^\cc_{k,\w}
\end{align*}
we also define
\begin{equation*}
   \tau
   = \sup_{\w \in \W}\, (\tau_\w^-+\tau_\w^+).
\end{equation*}
Consider the measurable RDS $\Psi\colon \Z\x\W\x X\to X$ given by
\begin{equation}\label{eq:Psi}
   \Psi_\w^n(x)=
   \begin{cases}
   \Phi_{\w}^n x + \dsum_{k=0}^{n-1} \Phi_{\theta^{k+1}\w}^{n-k-1} f_{\theta^k\w}(\Psi_{\w}^k(x))&\,\text{ if } n\geq1\\
      x&\,\text{ if } n=0\\
   \Phi_{\w}^n x - \dsum_{k=n}^{-1} \Phi_{\theta^{k+1}\w}^{n-k-1} f_{\theta^k\w}(\Psi_{\w}^k(x))&\,\text{ if } n\leq-1
   \end{cases}
\end{equation}
which encapsulates the solutions of the random nonlinear difference equation
\begin{equation*}
      x_{n+1}=\Phi_{\theta^n\w}^1x_n+f_{\theta^n\w}(x_n).
\end{equation*}

\begin{theorem} \label{thm:global:disc}
   Let $\Phi$ be a measurable linear RDS exhibiting an $\alpha$-trichotomy and let $f \in \sF$. If
   \begin{equation*}
     \lim_{n \to -\infty} \alpha^\sss_{-n,\theta^n \w} \alpha^\cc_{n,\w} = \lim_{n \to +\infty} \alpha^\uu_{-n,\theta^n \w} \alpha^\cc_{n,\w} = 0
   \end{equation*}
  for all $\w \in \W$, and
   \begin{equation*}
      \sgm +  \tau < 1/2,
   \end{equation*}
   then there are $N \in \ ]0,1[$ and a unique $\phi \in \cL_N$ such that for the RDS  $\Psi$ given by~\eqref{eq:Psi} we have
   \begin{equation}\label{eq:thm:global:invariance:discrete}
      \Psi_{\w}^n(\cV_{\phi,\w})
      \subseteq \cV_{\phi,\theta^n\w}
   \end{equation}
 for all $(n,\w) \in \Z \x \W$.  Moreover,  for every $(n, \w, \xi), (n, \w, \xi') \in \cH$ we have 
   \begin{equation*}
      \|\Psi_{\w}^n(\xi,\phi_\w(\xi)) - \Psi_{\w}^n(\xi',\phi_\w(\xi'))\|
      \le ({N}/{\tau}) \alpha^\cc_{n,\w} \, \|\xi - \xi'\|.
   \end{equation*}
  
\end{theorem}

The proof of Theorem~\ref{thm:global:disc} is analogous to the proof of Theorem~\ref{thm:global}. Therefore, in the remainder of this section, we provide a guide to the necessary adaptations.
 Fix $M$ and $N$ as in~\eqref{def:M+N}. Given $\w\in \W$ and $$x_\w=(x_\w^\cc,x_\w^\sss,x_\w^\uu)\in E_\w^\cc\x E_\w^\sss\x E_\w^\uu,$$
 the trajectory
 $$x_{\theta^n\w}=\Psi_\w^nx_\w=\prts{x_{\theta^n\w}^\cc,x_{\theta^n\w}^\sss,x_{\theta^n\w}^\uu}\in E_\w^\cc\x E_\w^\sss\x E_\w^\uu$$ satisfies the following equations for each $i\in\{\cc,\sss,\uu\}$:
\begin{equation}      \label{eq:dyn-split1a:discrete}
 x_{\theta^n\w}^i=\begin{cases}
   \Phi_{\w}^{i,n} x_\w^i + \dsum_{k=0}^{n-1} \Phi_{\theta^{k+1}\w}^{i,n-k-1} f_{\theta^k\w}(x_{\theta^k\w}^\cc,x_{\theta^k\w}^\sss,x_{\theta^k\w}^\uu)&\,\text{ if } n\geq 1 \\[5mm]
\Phi_{\w}^{i,n} x_\w^i - \dsum_{k=n}^{-1} \Phi_{\theta^{k+1}\w}^{i,n-k-1} f_{\theta^k\w}(x_{\theta^k\w}^\cc,x_{\theta^k\w}^\sss,x_{\theta^k\w}^\uu)&\,\text{ if } n\leq -1.
   \end{cases}
\end{equation}
In view of the invariance required in \eqref{eq:thm:global:invariance:discrete},  if $x_{\w}\in\cV_{\phi,\w}$ then 
$x_{\theta^n\w}$   must be in
$\cV_{\phi,\theta^n\w}$ for every $n \in \Z$, and thus, in this situation, the
equations from~\eqref{eq:dyn-split1a:discrete} can be written as
\begin{align}
x_{\theta^n\w}^\cc&=\begin{cases}
   \Phi_{\w}^{\cc,n} x_\w^\cc + \dsum_{k=0}^{n-1} \Phi_{\theta^{k+1}\w}^{\cc,n-k-1} f_{\theta^k\w}(x_{\theta^k\w}^\cc,\varphi_{\theta^k\w}(x_{\theta^k\w}^\cc))&\,\text{ if } n\geq 1 \\[5mm]
\Phi_{\w}^{\cc,n} x_\w^\cc - \dsum_{k=n}^{-1} \Phi_{\theta^{k+1}\w}^{\cc,n-k-1} f_{\theta^k\w}(x_{\theta^k\w}^\cc,\varphi_{\theta^k\w}(x_{\theta^k\w}^\cc))&\,\text{ if } n\leq -1
   \end{cases}
         \label{eq:dyn-split2a:discrete}
\end{align}
and, for $j\in\{\sss,\uu\}$,
\begin{align}
   \phi_{\theta^n\w}^j(x_{\theta^n\w})  &=\begin{cases}
   \Phi_{\w}^{j,n}  \phi_{\w}(x_\w) +\dsum_{k=0}^{n-1} \Phi_{\theta^{k+1}\w}^{j,n-k-1}f_{\theta^{k}\w}(x_{\theta^{k}\w}^\cc,\phi_{\theta^k\w}(x_{\theta^{k}\w}^\cc))&\,\text{ if } n\geq 1 \\[5mm]
\Phi_{\w}^{j,n}  \phi_{\w}(x_\w) -\dsum_{k=n}^{-1} \Phi_{\theta^{k+1}\w}^{j,n-k-1}f_{\theta^{k}\w}(x_{\theta^{k}\w}^\cc,\phi_{\theta^k\w}(x_{\theta^{k}\w}^\cc))&\,\text{ if } n\leq -1
   \end{cases}.
      \label{eq:dyn-split2b:discrete}
\end{align}
Let us prove prove that equations~\eqref{eq:dyn-split2a:discrete} and~\eqref{eq:dyn-split2b:discrete} have
solutions.
First, we rewrite them, by a discrete version of Lemma~\ref{lemma:equiv}.
\begin{lemma} \label{lemma:equiv:discrete}
   Consider $(h,\phi) \in \cX_{M,N}$ such that, for all $(n,\w) \in \Z\x\W$ and all $\xi \in E_\w^\cc$
   \begin{align}
h_{n,\w}(\xi)&=\begin{cases}
   \Phi_{\w}^{\cc,n} \xi + \dsum_{k=0}^{n-1} \Phi_{\theta^{k+1}\w}^{\cc,n-k-1} f_{\theta^k\w}(h_{k,\w}(\xi),\varphi_{\theta^k\w}(h_{k,\w}(\xi)))&\,\text{ if } n\geq 1 \\[5mm]
\Phi_{\w}^{\cc,n} \xi - \dsum_{k=n}^{-1} \Phi_{\theta^{k+1}\w}^{\cc,n-k-1} f_{\theta^k\w}(h_{k,\w}(\xi),\varphi_{\theta^k\w}(h_{k,\w}(\xi)))&\,\text{ if } n\leq -1
   \end{cases}.
               \label{eq:dyn-split3a:discrete}
   \end{align}
Then the following conditions $a)$ and $b)$ are equivalent:
   \begin{enumerate}[\lb=$\alph*)$,\lm=5mm]
      \item For each $j\in\{\uu,\sss\}$ and all $(n,\w,\xi) \in \cH$
       \begin{align}
   \phi_{\theta^n\w}^j(h_{n,\w}(\xi))  &=\begin{cases}
   \Phi_{\w}^{j,n}  \phi_{\w}(\xi) +\dsum_{k=0}^{n-1} \Phi_{\theta^{k+1}\w}^{j,n-k-1}f_{\theta^{k}\w}(h_{k,\w}(\xi),\phi_{\theta^k\w}(h_{k,\w}(\xi)))&\,\text{ if } n\geq 1 \\[5mm]
\Phi_{\w}^{j,n}  \phi_{\w}(\xi) -\dsum_{k=n}^{-1} \Phi_{\theta^{k+1}\w}^{j,n-k-1}f_{\theta^{k}\w}(h_{k,\w}(\xi),\phi_{\theta^k\w}(h_{k,\w}(\xi)))&\,\text{ if } n\leq -1
   \end{cases}
            \label{eq:dyn-split3b:discrete}
        \end{align}
      \item For all $(\w,\xi) \in \cG$
      \begin{equation} \label{eq:phi_n:discrete_ss}
            \phi_{\w}^\sss(\xi)
               = \sum_{k=-\infty}^{-1} \Phi_{\theta^{k+1}\w}^{\sss,-(k+1)} f_{\theta^{k}\w}(h_{k,\w}(\xi),\phi_{\theta^{k}\w}(h_{k,\w}(\xi))).
        \end{equation}
        and
        \begin{equation} \label{eq:phi_n:discrete_uu}
            \phi_{\w}^\uu(\xi)
               = -\sum_{k=0}^{+\infty} \Phi_{\theta^{k+1}\w}^{\uu,-(k+1)} f_{\theta^{k}\w}(h_{k,\w}(\xi),\phi_{\theta^{k}\w}(h_{k,\w}(\xi)))   .         
        \end{equation}
   \end{enumerate}
\end{lemma}

Consider here the operator \( \JJ \), which assigns each pair \( (h, \phi) \in \cX^{(B)}_{M,N} \) to the map \( \JJ(h, \phi) \colon \cH \to X \) given by
 \begin{align*}
   \prtsr{\JJ\prts{h,\phi}}(n,\w,\xi)&=\begin{cases}
   \Phi_{\w}^{\cc,n} \xi + \dsum_{k=0}^{n-1} \Phi_{\theta^{k+1}\w}^{\cc,n-k-1} f_{\theta^k\w}(h_{k,\w}(\xi),\varphi_{\theta^k\w}(h_{k,\w}(\xi)))&\,\text{ if } n\geq 1 \\[5mm]
\Phi_{\w}^{\cc,n} \xi - \dsum_{k=n}^{-1} \Phi_{\theta^{k+1}\w}^{\cc,n-k-1} f_{\theta^k\w}(h_{k,\w}(\xi),\varphi_{\theta^k\w}(h_{k,\w}(\xi)))&\,\text{ if } n\leq -1
   \end{cases}
 \end{align*}
and $\LL$ be the operator that assigns to each pair $(h,\phi)\in\cX_{M,N}$ the map $\LL(h,\phi)\colon \cG \to X$ defined by
	\[
	\prtsr{\LL\prts{h,\phi}}(\w,\xi)=\prtsr{\LL^\sss\prts{h,\phi}}(\w,\xi)+\prtsr{\LL^\uu\prts{h,\phi}}(\w,\xi),
	\]
where
\begin{align*}
   \prtsr{\LL^\sss\prts{h,\phi}}(\w,\xi)
   & = \sum_{k=-\infty}^{-1} \Phi_{\theta^{k+1}\w}^{\sss,-(k+1)} f_{\theta^{k}\w}(h_{k,\w}(\xi),\phi_{\theta^{k}\w}(h_{k,\w}(\xi)))
\end{align*}
and
\begin{align*}
   \prtsr{\LL^\uu\prts{h,\phi}}(\w,\xi)
   & = -\sum_{k=0}^{+\infty} \Phi_{\theta^{k+1}\w}^{\uu,-(k+1)} f_{\theta^{k}\w}(h_{k,\w}(\xi),\phi_{\theta^{k}\w}(h_{k,\w}(\xi))).
\end{align*}

To finalize, define $\TT \colon \cX_{M,N}  \to \cX_{M,N} $ by
   $$ \TT(h,\phi) = \prts{\JJ(h,\phi), \LL(h,\phi)}.$$
The operator $\TT$ is a contraction in $(\cX_{M,N},d)$. By the Banach Fixed Point Theorem, $\TT$ as a unique fixed point $(h,\phi)$, which satisfies conditions~\eqref{eq:dyn-split3a:discrete}, \eqref{eq:phi_n:discrete_ss} and~\eqref{eq:phi_n:discrete_uu}. By Lemma~\ref{lemma:equiv:discrete} the pair $(h,\phi)$ also satisfy the conditions in~\eqref{eq:dyn-split3b:discrete}. Hence, by~\eqref{eq:dyn-split2a:discrete} and~\eqref{eq:dyn-split2b:discrete}, we get that $\prts{h_{n,\w}(\xi), \phi_{\theta^n\w}(h_{n,\w}(\xi))}$ is the orbit by $\Psi$ of the initial condition $$(\xi,\phi_\w^\sss(\xi),\phi_\w^\uu(\xi)) \in E_\w^\cc \x E_\w^\sss \x E_\w^\uu.$$ The graphs $\cV_{\phi,\w}$ are the required invariant manifolds of $\Psi$. Furthermore, for all $\w\in \W$, all $n\in\Z$ and  all $\xi,\xi'\in E_\w^\cc$ it follows from~\eqref{eq:||phi_w(xi)-pi_w(bxi)||<=...}, \eqref{ine:||h_n,w(xi)-h_n,w(bxi)||<=M...} and~\eqref{def:M+N} that
   \begin{equation*}
      \|\Psi_{\w}^n(\xi,\phi_\w(\xi)) - \Psi_{\w}^n(\xi',\phi_\w(\xi'))\|
         \le \frac{N}{\tau} \alpha^\cc_{n,\w} \|\xi-\bar \xi\|,
   \end{equation*}
which finishes the proof of  Theorem~\ref{thm:global:disc}.

\section{Continuous-time examples}\label{section:examples:cont}

For this section assume $\T=\R$. Throughout this entire section we consider a constant $\delta \in\, ]0,1/6[$ and a random variable $G \colon \W \to \, ]0,+\infty[$ satisfying
\begin{equation*}
   \int_{-\infty}^{+\infty} G(\theta^r\w)\,dr \le 1
   \ \ \ \text{ for all } \w \in \W.
\end{equation*}
 In all the following examples we may consider different growth rates along the \textit{central directions} $E_\w^\cc$, depending if we are looking to the \textit{future} ($t\to+\infty$) or to the \textit{past} ($t\to-\infty$).

\subsection{Tempered exponential trichotomies}
   Let
   \begin{equation*}
      \lbd^\cs, \lbd^\cu, \lbd^\sss, \lbd^\uu \colon \W \to \R
   \end{equation*}
   be $\theta$-invariant random variables, i.e. satisfying $\lbd^\ell(\theta^t \w) = \lbd^\ell(\w)$  for all $\w\in\W$, $t\in\R$ and $\ell\in\set{\cs, \cu, \sss, \uu}$.
   A Bochner measurable linear RDS $\Phi$ exhibits an \emph{exponential trichotomy} if it exhibits a generalized trichotomy with bounds
   \begin{align*}
      &    
         \alpha^\cc_{t,\w} =
         \begin{cases}
            K(\w) \e^{\lbd^\cs(\w)t}, \,\, t\ge0\\
            K(\w) \e^{\lbd^\cu(\w)t}, \,\, t\le0
         \end{cases}\\
      &   
         \alpha^\sss_{t,\w} = \,\,\,\, K(\w) \e^{\lbd^\sss(\w)t}, \,\, t\ge0\\
      &   
         \alpha^\uu_{t,\w} = \,\,\,\, K(\w) \e^{\lbd^\uu(\w)t}, \,\, t \le 0
   \end{align*}
   for some random variable $K:\W \to [1,+\infty[$.  If the  random variable $K$ is \emph{tempered}, \emph{i.e.}, if
\begin{equation}\label{eq:tempered}
    \Lambda_{K,\gamma,\w}:=\sup_{t \in \T} \prtsr{\e^{-\gamma |t|}  K(\theta^t w)} < +\infty
\end{equation}
for all $\gamma >0$ and all $\w \in \W$, we say that $\Phi$ exhibits an \emph{tempered exponential trichotomy}. 
   
\begin{corollary}
   \label{corollary:tempered:cont}
    Let $\Phi$ be a Bochner measurable linear RDS exhibiting a tempered exponential trichotomy such that
   \begin{equation*}
      \lbd^\cu(\w)>\lbd^\sss(\w) 
      \qquad\text{ and }\qquad
      \lbd^\cs(\w)<\lbd^\uu(\w) 
   \end{equation*}
   for all $\w \in \W$, and let $f \in \sF^{(B)}_\alpha$. Assume that $\Psi$ is a Bochner measurable RDS such that~\eqref{eq:u(t)=...} has  unique solution $\Psi(\cdot,\w,x)$ for every $(\w,x) \in \W\x X$. Consider a $\theta$-invariant random variable $\gamma(\w)>0$ satisfying
   \begin{equation*}
       a(\w):=\lbd^\cu(\w)-\lbd^\sss(\w)-\gamma(\w)>0
       \quad\text{ and } \quad 
       b(\w):=\lbd^\uu(\w)-\lbd^\cs(\w)-\gamma(\w)>0
   \end{equation*}
   for all $\w \in \W$. If
   \begin{equation*}
       \Lip(f_\w)\le \frac{\delta}{K(\w)}
         \mins{G(\w),
         \frac{a(\w)}{\Lambda_{K,\gamma(\w),\w}},
         \frac{b(\w)}{\Lambda_{K,\gamma(\w),\w}}}
   \end{equation*}
   for all $\w \in \W$, then the same conclusions of Theorem~\ref{thm:global} hold.
\end{corollary}

\begin{proof}
   Since $K$ is a tempered random variable, we have
   \begin{equation*}
   \begin{split}
       \lim_{t \to -\infty} \alpha^\sss_{-t,\theta^t \w} \alpha^\cc_{t,\w}
       &= \lim_{t \to -\infty} K(\w) K(\theta^t \w) \e^{(\lbd^\cu(\w) - \lbd^\sss(\w))t}\\
		&\leq  \lim_{t \to -\infty} K(\w) \Lambda_{K,a(\w),\w} \e^{\gamma(\w)t}
       = 0
       \end{split}
   \end{equation*}
   and
   \begin{equation*}
   \begin{split}
       \lim_{t \to +\infty} \alpha^\uu_{-t,\theta^t \w} \alpha^\cc_{t,\w}
       &= \lim_{t \to +\infty} K(\w) K(\theta^t \w) \e^{(\lbd^\cs(\w) - \lbd^\uu(\w))t}\\
      &\leq  \lim_{t \to +\infty} K(\w) \Lambda_{K,b(\w),\w} \e^{\gamma(\w)t}
       = 0
       \end{split}
   \end{equation*}
   for all $\w \in \W$. Therefore condition~\eqref{eq:CondicaoTeo} holds. Let us check ~\eqref{eq:CondicaoTeo:alpha+tau<1/2}. Indeed, for every $t \ge 0$ and every $\w \in \W$ we have
   \begin{align*}
         \dfrac{1}{\alpha^\cc_{t,\w}} \int_{0}^{t} \alpha^\cc_{t-r,\theta^{r}\w}
            \Lip(f_{\theta^r\w}) \alpha^\cc_{r,\w}\dr
      & 
         = \int_{0}^{t} K(\theta^{r} \w) \Lip(f_{\theta^r \w})\dr\\
      & 
         \le \delta \int_{-\infty}^{+\infty} G(\theta^r\w)\dr\\
      & 
         \le \delta,
   \end{align*}
   and, similarly, for every $t \le 0$ and every $\w \in \W$ we have
\begin{equation*}
         \dfrac{1}{\alpha^\cc_{t,\w}} \int_{t}^{0} \alpha^\cc_{t-r,\theta^{r}\w}
            \Lip(f_{\theta^r\w}) \alpha^\cc_{r,\w}\dr
      \le \delta.
         \end{equation*}
Thus, $\sigma\leq\delta$.   Moreover, since $K(\w)\leq \e^{\gamma(\w) \abs{r}} \Lambda_{K,\gamma(\w),\theta^r\w}$ for every $\w\in\W$ and $r\in\R$, we have
   \begin{align*}      
         \dint_{-\infty}^{0} \alpha^\sss_{-r,\theta^r\w} \Lip(f_{\theta^r\w})
            \alpha^\cc_{r,\w} \dr
      & 
         = \dint_{-\infty}^{0} K(\w) K(\theta^r\w) \e^{(\lbd^\cu(\w) - \lbd^\sss(\w))r}
            \Lip(f_{\theta^r\w}) \dr\\
      & 
         \le \delta \dint_{-\infty}^{0}  
           a(\w)
            \e^{a(\w)r} \dr\\
      & 
         \le \delta.
   \end{align*}
   and
   \begin{align*}
         \dint_{0}^{+\infty} \alpha^\uu_{-r,\theta^r\w} \Lip(f_{\theta^r\w})
            \alpha^\cc_{r,\w} \dr
      &
         = \dint_{0}^{+\infty} K(\w) K(\theta^r\w) \e^{(\lbd^\cs(\w) - \lbd^\uu(\w))r}
            \Lip(f_{\theta^r\w}) \dr\\
      &
         \le \delta \dint_{0}^{+\infty}
            b(\w)
            \e^{-b(\w)r} \dr\\
      &
         \le \delta.
   \end{align*} 
 Henceforth, $\sigma+\tau \le 3\delta < 1/2$.
\end{proof}

\subsection{$\psi$-trichotomies}

   Consider measurable functions
      $$ \psi^{\cs},  \psi^{\cu}, \psi^{\sss},   \psi^{\uu}\colon \T \x \W \to ]0,+\infty[$$
   such that for $\ell\in\{\cs, \cu, \sss, \uu\}$ we have
   \begin{equation}\label{eq:psi:cociclo}
      \psi^\ell(t+s,\w) = \psi^\ell(t,\theta^s \w) \psi^\ell (s,\w)
   \end{equation}
   for all $t,s \in \T$ and all $\w \in \W$.
     A \emph{$\psi$-trichotomy} is a generalized trichotomy with bounds
   \begin{equation*}
   \begin{split}
      &    
         \alpha^\cc_{t,\w} =
         \begin{cases}
            K(\w) \psi^\cs(t,\w), \,\, t\ge0\\
            K(\w) \psi^\cu(t,\w), \,\, t\le0
         \end{cases}\\
      &   
         \alpha^\sss_{t,\w} = \,\,\,\, K(\w) \psi^\sss(t,\w), \,\, t\ge0\\
      &   
         \alpha^\uu_{t,\w} = \,\,\,\, K(\w) \psi^\uu(t,\w), \,\, t \le 0   \end{split}
   \end{equation*}
      for a random variable $K\colon\W\to[1,+\infty[$. 
      For all $\ell\in\{\cu,\cs,\uu,\sss\}$ set
   \begin{equation}\label{eq:derivative:psi}
d_{\psi^\ell}(\w)=\lim_{h\to0}\dfrac{\psi^\ell(h,\w)-1}{h}.
   \end{equation}
Since $\psi^\ell(0,\w)=1$, from~\eqref{eq:psi:cociclo} we have
\[
\frac{d}{dt}\psi^\ell(t,\w)=d_{\psi^\ell}(\theta^t\w)\psi^\ell(t,\w)
\]
whenever limits~\eqref{eq:derivative:psi} exist.
Moreover, in this situation we also have
\[
\frac{d}{dt}\psi^\ell(-t,\theta^t\w)=\frac{d}{dt}\frac{1}{\psi^\ell(t,\w)}=-d_{\psi^\ell}(\theta^t\w)\psi^\ell(-t,\theta^t\w).
\]

 From now on we also assume that for all $\w\in\W$ the following limit exists:
   \begin{equation}\label{eq:derivative:K}
d_K(\w)=\lim_{h\to0}\dfrac{K(\theta^{h}\w)-K(\w)}{h}.
   \end{equation}
We notice that for all $t\in\R$, $\frac{d}{dt}K(\theta^t\w)=d_K(\theta^t\w)$.

\begin{corollary}\label{corollary:psi:cont}   Let $\Phi$ be a Bochner measurable linear RDS exhibiting a $\psi$-trichotomy such that the limits in~\eqref{eq:derivative:psi} and~\eqref{eq:derivative:K} exist and satisfy
   \[    d_{\psi^\cs}(\w)-d_{\psi^\uu}(\w)<  \dfrac{d_K(\w)}{K(\w)} < d_{\psi^\cu}(\w)-d_{\psi^\sss}(\w)
   \]
   for all $\w \in \W$. Let $f \in \sF^{(B)}_\alpha$ be such that
   \begin{equation*}
      \Lip(f_\w)
      \le \dfrac{\delta}{K(\w)}
         \mins{G(\w), \dfrac{a(\w)}{K(\w)}, \dfrac{b(\w)}{K(\w)}}
   \end{equation*}
   for all $\w \in \W$, where
 \[
   a(\w)=\dfrac{d_K(\w)}{K(\w)} -d_{\psi^\cs}(\w)+d_{\psi^\uu}(\w)\,\text{ and }\, b(\w)=-\dfrac{d_K(\w)}{K(\w)} +d_{\psi^\cu}(\w)-d_{\psi^\sss}(\w).
   \]
Assume that $\Psi$ is a Bochner measurable RDS such that~\eqref{eq:u(t)=...} has unique solution $\Psi(\cdot,\w,x)$ for every $\w \in \W$ and every $x \in X$. If, for all $\w \in \W$,
   \begin{equation}\label{eq:corol:psi:limits}
          \dlim_{t \to -\infty} K(\theta^t \w) \psi^\sss(-t,\theta^t\w)\psi^\cu(t,\w) =\dlim_{t \to +\infty} K(\theta^t \w) \psi^\uu(-t,\theta^t\w)\psi^\cs(t,\w) = 0
      \end{equation}
   then the same conclusions of Theorem~\ref{thm:global} hold.
\end{corollary}

\begin{proof}
  Conditions in~\eqref{eq:corol:psi:limits} are equivalent to those in~\eqref{eq:CondicaoTeo}, and, as in the proof of Corollary~\ref{corollary:tempered:cont} we have
  $\sgm \le \delta$. Moreover, since
   \begin{align*}
      &
         \frac{d}{dt}
         \prts{\dfrac{\psi^\uu(-t,\theta^t\w)\psi^\cs(t,\w)}
         {K(\theta^{t}\w)}}\\
      &
         = \dfrac{\prts{-d_{\psi^\uu}(\theta^t\w)+d_{\psi^\cs}(t,\w)}K(\theta^t\w)-d_K(\theta^t\w)}
            {[K(\theta^t\w)]^2}\,\psi^\uu(-t,\theta^t\w)\,\psi^\cs(t,\w)\\
            &
         =-\dfrac{a(\theta^t\w)}{K(\theta^t\w)}\,\psi^\uu(-t,\theta^t\w)\,\psi^\cs(t,\w)
         ,
   \end{align*}
   we have
   \begin{align*}
      & \dint_{0}^{+\infty}
         \alpha^\uu_{-r,\theta^r\w} \Lip(f_{\theta^r\w}) \alpha^\cs_{r,\w} \dr\\
      & = K(\w) \dint_{0}^{+\infty}
         K(\theta^r\w)\psi^\uu(-r,\theta^r \w) \Lip(f_{\theta^r\w})\psi^\cs(r,\w) \dr
         \\
      & \le \delta K(\w)\dint_{0}^{+\infty}
         \dfrac{a(\theta^r\w)}{K(\theta^r\w)}\,\psi^\uu(-r,\theta^r\w)\,\psi^\cs(r,\w)\,dr\\
      & = \delta
         - \delta K(\w)\lim_{r \to +\infty}
         \dfrac{\psi^\uu(-r,\theta^r\w)\psi^\cs(r,\w)
         }{K(\theta^{r}\w)}\\
            & = \delta.
   \end{align*}
   Similarly, since
   \begin{align*}
      &
         \frac{d}{dt}
         \prts{\dfrac{\psi^\sss(-t,\theta^t\w)\psi^\cu(t,\w)}
         {K(\theta^{t}\w)}}\\
      &
         = \dfrac{\prts{-d_{\psi^\sss}(\theta^t\w)+d_{\psi^\cu}(t,\w)}K(\theta^t\w)-d_K(\theta^t\w)}
            {[K(\theta^t\w)]^2}\,\psi^\sss(-t,\theta^t\w)\,\psi^\cu(t,\w)\\
            &
         =\dfrac{b(\theta^t\w)}{K(\theta^t\w)}\,\psi^\sss(-t,\theta^t\w)\,\psi^\cu(t,\w)
         ,
   \end{align*}
   we have
   \begin{align*}
      & \dint_{-\infty}^{0}
         \alpha^\sss_{-r,\theta^r\w} \Lip(f_{\theta^r\w}) \alpha^\cu_{r,\w} \dr\\
      & = K(\w) \dint_{-\infty}^{0}
         K(\theta^r\w)\psi^\sss(-r,\theta^r \w) \Lip(f_{\theta^r\w})\psi^\cu(r,\w) \dr
         \\
      & \le \delta K(\w)\dint_{-\infty}^{0}
         \dfrac{b(\theta^r\w)}{K(\theta^r\w)}\,\psi^\sss(-r,\theta^r\w)\,\psi^\cu(r,\w)\,dr\\
      & = \delta
         - \delta K(\w)\lim_{r \to -\infty}
        \dfrac{\psi^\sss(-r,\theta^r\w)\psi^\cu(r,\w)
         }{K(\theta^{r}\w)}\\
            & = \delta.
   \end{align*}
   
       Thus $\sigma+\tau \le 3\delta < 1/2$.
\end{proof}

In the following we provide a particular example of a $\psi$-trichotomy in $\R^4$.
      
\begin{example} Let $\psi^{\cs},  \psi^{\cu}, \psi^{\sss},   \psi^{\uu}\colon \T \x \W \to ]0,+\infty[$ be measurable functions and let $K \colon \W \to [1,+\infty[$ be a random variable. In $X=\R^4$, equipped with the maximum norm, consider the  projections
   \begin{align*}
      P_\w^\cs(x_1,x_2,x_3,x_4)&=(0,0,x_3+(K(\w)-1)x_4, 0)\\
      P_\w^\cu(x_1,x_2,x_3,x_4)&=(
      (1-K(\w))x_2,x_2,0,0)\\
      P_\w^\sss(x_1,x_2,x_3,x_4)&=(x_1+(K(\w)-1)x_2,0, 0,0)\\
      P_\w^\uu(x_1,x_2,x_3,x_4)&=(0,0,(1-K(\w))x_4,x_4)
   \end{align*}
For all $\w',\w\in\W$,
   \begin{align*}
    P_{\w'}^\cs P_\w^\uu &= (0,0,(K(\w')-K(\w))x_4,0)\\
        P_{\w'}^\sss P_\w^\cu &= ((K(\w')-K(\w))x_2,0,0,0)
     \end{align*}
and for all the remaining $i,j\in\{{\cs}, \cu, {\sss}, {\uu}\}$, with $i \neq j$,
   \begin{equation*}
      P_{\w'}^i P_\w^j =0.
    \end{equation*}
Notice that for all $\w, {\w'} \in \W$
   \begin{equation*}
      P_{\w'}^\sss  P_\w^\sss=P_\w^\sss, \,\,\,
      P_{\w'}^\uu  P_\w^\uu=P_{\w'}^\uu, \,\,\,
      P_{\w'}^\cs  P_\w^\cs=P_\w^\cs\,\,\,
      \text{ and }\,\,\,
      P_{\w'}^\cu  P_\w^\cu=P_{\w'}^\cs.
   \end{equation*}
Moreover,
\begin{equation*}
\|P_\w^\cs\|=\|P_\w^\sss\|=K(\w)
\end{equation*}
and
\begin{equation*}
\|P_\w^\cu\|=\|P_\w^\uu\|=\max\set{K(\w)-1,1}\leq K(\w).
\end{equation*}

We define $\Phi\colon\T\x\W\x \R^4\to\R^4$  by
   \begin{equation*}
      \Phi_\w^t
      = \psi^\cs(t,\w) \, P_\w^\cs
         + \dfrac{K(\w)}{K(\theta^t \w)} {\psi^{\cu}(t,\w)} \, P_{\theta^t \w}^\cu
         +\psi^\sss(t,\w)P_\w^\sss+ \dfrac{K(\w)}{K(\theta^t \w)} {\psi^{\uu}(t,\w)} \, P_{\theta^t \w}^\uu.
   \end{equation*}
  Let $P^\cc=P^\cs+P^\cu$ and $\mathcal P=(P^\cc,P^\sss,P^\uu)$. We have that $\Phi$ is a measurable linear RDS over $\Sigma$ that admits a measurable $\mathcal P$-invariant splitting, and
   \begin{align*}
     \|\Phi^{\cc,t}_\w \| &= \max\set{\psi^\cs(t,\w)\|P_\w^\cs\|, \frac{1}{K(\theta^t\w)}\psi^\cu(t,\w)\|P_{\theta^t\w}^\cu\|}\\
     &\leq  K(\w) \max\set{\psi^\cs(t,\w),\psi^\cu(t,\w)}\\
     \|\Phi^{\sss,t}_\w \| &= \psi^\sss(t,\w)\|P_\w^\sss\|= K(\w) \psi^\sss(t,\w)\\
      \|\Phi^{\uu,t}_{\w} \| &= \dfrac{K(\w)}{K(\theta^t\w)}\psi^\uu(t,\w)\|P_{\theta^t\w}^\uu\|\leq K(\w) \psi^\uu(t,\w)
       \end{align*}
          Hence the linear RDS $\Phi$ exhibits a generalized trichotomy with bounds
        \begin{align*}
     \alpha^\cc_{t,\w}&=  K(\w)\max\set{ \psi^\cs(t,\w),\psi^\cu(t,\w) }\\
     \alpha^\sss_{t,\w}&= K(\w) \psi^\sss(t,\w)\\
      \alpha^\uu_{t,\w}&= K(\w) \psi^\uu(t,\w)
       \end{align*}
If we assume $\psi^\cs(t,\w)\ge\psi^\cu(t,\w)$ for all $t\ge 0$ then
\begin{equation*}
     \alpha^\cc_{t,\w}=
     \begin{cases}
     K(\w)\psi^\cs(t,\w)&\text{ if } t\ge0\\
     K(\w)\psi^\cu(t,\w)&\text{ if } t\le0
     \end{cases}
\end{equation*}
and $\Phi$ exhibits a $\psi$-trichotomy.
      \end{example}
      
In the next sections we consider particular $\psi$-trichotomies.

\subsubsection{Integral exponential trichotomy}
  \label{sec:ex:integral}
   Let
   \begin{equation*}
      \lbd^\cs, \lbd^\cu, \lbd^\sss, \lbd^\uu \colon \W \to \R
   \end{equation*}
   be random variables such that for all $\w\in\W$ and $\ell\in\set{\cs,\cu,\sss,\uu}$, the map $r \mapsto \lbd^\ell(\theta^r \w)$ is integrable in every interval $[0,t]$
An \emph{integral exponential trichotomy} is a $\psi$-trichotomy with
   \[
   \psi^\ell(t,\w)=\e^{\int_0^t\lbd^\ell(\theta^r\w)dr}
   \]
for all $\ell\in\set{\cs, \cu, \sss, \uu}$, i.e., is a generalized trichotomy with bounds
   \begin{equation*}
   \begin{split}
      &    
         \alpha^\cc_{t,\w} =
         \begin{cases}
            K(\w) \e^{\int_0^t\lbd^\cs(\theta^r\w)dr}, \,\, t\ge0\\
            K(\w) \e^{\int_0^t\lbd^\cu(\theta^r\w)dr}, \,\, t\le0
         \end{cases}\\
      &   
         \alpha^\sss_{t,\w} = \,\,\,\, K(\w) \e^{\int_0^t\lbd^\sss(\theta^r\w)dr}, \,\, t\ge0\\
      &   
         \alpha^\uu_{t,\w} = \,\,\,\, K(\w) \e^{\int_0^t\lbd^\uu(\theta^r\w)dr}, \,\, t \le 0.   \end{split}
   \end{equation*}
Notice that  if
\begin{equation}\label{eq:lbd:integral}
\dlim_{h\to0}\frac1h\int_0^h\lbd^\ell(\theta^r\w)\,dr=\lbd^\ell(\w)
\end{equation}
then $d_{\psi^\ell}(\w)=\lbd^\ell(\w)$ for all $\ell\in\set{\cs, \cu, \sss, \uu}$. From Corollary~\ref{corollary:psi:cont} we get the following.

\begin{corollary}
     Let $\Phi$ be a Bochner measurable linear RDS exhibiting
     an integral exponential trichotomy such that \eqref{eq:lbd:integral} holds and the limit~\eqref{eq:derivative:K} exists and satisfyies
   \[    \lbd^\cs(\w)-\lbd^\uu(\w)<  \dfrac{d_K(\w)}{K(\w)} < \lbd^\cu(\w)-\lbd^\sss(\w)
   \]
   for all $\w \in \W$.
   Let $f \in \sF^{(B)}_\alpha$ be such that
   \begin{equation*}
      \Lip(f_\w)
      \le \dfrac{\delta}{K(\w)}
         \mins{G(\w), \dfrac{a(\w)}{K(\w)}, \dfrac{b(\w)}{K(\w)}}
   \end{equation*}
   for all $\w \in \W$, where
   \[
   a(\w)=\dfrac{d_K(\w)}{K(\w)} -\lbd^\cs(\w)+\lbd^\uu(\w)\,\text{ and }\,b(\w)=-\dfrac{d_K(\w)}{K(\w)} +\lbd^\cu(\w)-\lbd^\sss(\w).\]   
Assume that $\Psi$ is a Bochner measurable RDS such that~\eqref{eq:u(t)=...} has a unique solution $\Psi(\cdot,\w,x)$ for every $\w \in \W$ and every $x \in X$. If for all $\w \in \W$,
   \begin{equation*}
      \dlim_{t \to +\infty}
      K(\theta^t \w) \e^{\int_0^t \lbd^\cs(\theta^r \w)-\lbd^\uu(\theta^r \w) \dr}
      =
      \dlim_{t \to -\infty}
      K(\theta^t \w) \e^{\int_0^t \lbd^\cu(\theta^r \w)-\lbd^\sss(\theta^r \w)\dr}
      = 0
   \end{equation*}
   then the same conclusions of Theorem~\ref{thm:global} hold.
\end{corollary}

\subsubsection{Non exponential trichotomies}

We provide now a particular type of $\psi$-trichotomies that can be easily handled to construct trichotomies beyhond the exponential bounds. Let
   \begin{equation*}
      \lbd^\cs, \lbd^\cu, \lbd^\sss, \lbd^\uu \colon \W \to \R
   \end{equation*}
   be random variables such that for all $\ell\in\set{\cs, \cu, \sss, \uu}$ the following limit exists for all $\w$:
   \begin{equation}\label{eq:derivative:nonexp}
   d_{\lambda^\ell}(\w):=\lim_{h\to 0}\frac{\lambda^\ell(\theta^h\w)-\lambda^\ell(\w)}{h}.
   \end{equation}
   Consider a $\psi$-trichotomy with
   \[
   \psi^\ell(t,\w)=\dfrac{\lbd^\ell(\w)}{\lbd^\ell(\theta^t \w)}
   \]
for all $\ell\in\set{\cs, \cu, \sss, \uu}$, i.e., is a generalized trichotomy with bounds
   \begin{equation}\label{eq:bounds:trich:nonexp:continuous}
   \begin{split}
      &    
         \alpha^\cc_{t,\w} =
         \begin{cases}
             K(\w) \dfrac{\lbd^\cs(\w)}{\lbd^\cs(\theta^t \w)}, \,\, t\ge0\\[4mm] 
            K(\w) \dfrac{\lbd^\cu(\w)}{\lbd^\cu(\theta^t \w)}, \,\, t\le0
         \end{cases}\\
      &   
         \alpha^\sss_{t,\w} = \,\,\,\, K(\w) \frac{\lbd^\sss(\w)}{\lbd^\sss(\theta^t \w)}, \,\, t\ge0\\
      &   
         \alpha^\uu_{t,\w} = \,\,\,\, K(\w) \frac{\lbd^\uu(\w)}{\lbd^\uu(\theta^t \w)}, \,\, t \le 0.   \end{split}
   \end{equation}
Notice that 
\begin{equation*}
   d_{\psi^\ell}(\w)=-\frac{d_{\lbd^\ell}(\w)}{\lambda^\ell(\w)}.
\end{equation*}
 for all $\ell\in\set{\cs, \cu, \sss, \uu}$. 
From Corollary~\ref{corollary:psi:cont} we get the following.

\begin{corollary}
   Let $\Phi$ be a Bochner measurable linear RDS exhibiting an $\alpha$-trichotomy, with bounds~\eqref{eq:bounds:trich:nonexp:continuous} and such that~\eqref{eq:derivative:nonexp} and~\eqref{eq:derivative:K} exist and satisfy
   \[
\frac{d_{\lbd^\uu}(\w)}{\lambda^\uu(\w)}-\frac{d_{\lbd^\cs}(\w)}{\lambda^\cs(\w)}< \dfrac{d_K(\w)}{K(\w)} <\frac{d_{\lbd^\sss}(\w)}{\lambda^\sss(\w)}-\frac{d_{\lbd^\cu}(\w)}{\lambda^\cu(\w)}.
   \]
    Let $f \in \sF^{(B)}_\alpha$ be such that for all $\w \in \W$ we have
   \begin{equation*}
      \Lip(f_\w)
      \le \dfrac{\delta}{K(\w)}
         \mins{G(\w), \dfrac{a(\w)}{K(\w)}, \dfrac{b(\w)}{K(\w)}},
   \end{equation*}
    where
   \[
   a(\w)=\dfrac{d_K(\w)}{K(\w)} +\frac{d_{\lbd^\cs}(\w)}{\lambda^\cs(\w)}-\frac{d_{\lbd^\uu}(\w)}{\lambda^\uu(\w)}
   \]
   and
   \[b(\w)=-\dfrac{d_K(\w)}{K(\w)}-\frac{d_{\lbd^\cu}(\w)}{\lambda^\cu(\w)} +\frac{d_{\lbd^\sss}(\w)}{\lambda^\sss(\w)}.
   \]
   Assume that $\Psi$ is a Bochner measurable RDS such that~\eqref{eq:u(t)=...} has a unique solution $\Psi(\cdot,\w,x)$ for every $\w \in \W$ and every $x \in X$. If for all $\w \in \W$,
   \begin{equation*}
         \dlim_{t \to +\infty}  K(\theta^t \w)\dfrac{\lbd^\sss(\theta^t \w)}{\lbd^\cu(\theta^t\w)}=\dlim_{t \to +\infty}K(\theta^t \w)\dfrac{\lbd^\uu(\theta^t \w)}{\lbd^\cs(\theta^t\w)} = 0
      \end{equation*}
   then the same conclusions of Theorem~\ref{thm:global} hold.
\end{corollary}

\begin{example}[Nonexponential trichotomy]\label{example:Nonexponential trichotomy}
Consider for the driving system the horizontal flow in $\R^2$ given by $\,\theta^t (x,y) = (x+t,y)$, which preserves the Lebesgue measure. Let 
   $C, \xi^\cs, \xi^\cu,\xi^\sss, \xi^\uu$ and $\eps$ be some real constants with $C \ge 1$ and  $\eps \ge 0$, and set:
   \begin{align*}
      & \lbd^\ell(x,y)=(1+x^2)^{-(1+y^2)\xi_\ell},\,\,\ell\in\set{\cs,\cu,\sss,\uu}\\
      & K(x,y)=C (1+x^2)^{(1+y^2)\eps}.
 \end{align*}
In this case we obtain a polynomial type trichotomy. Let us assume $\lbd_\cs\ge\lbd_\cu$. Thus we have a trichotomy with
\[
\begin{split}
      &    
         \alpha^\cc_{t,(x,y)} =
         \begin{cases}
             C\pfrac{1+(x+t)^2}{1+x^2}^{(1+y^2)\xi^\cs} (1+x^2)^{(1+y^2)\eps},& t\ge0\\[4mm] 
            C\pfrac{1+(x+t)^2}{1+x^2}^{(1+y^2)\xi^\cu} (1+x^2)^{(1+y^2)\eps}, & t\le0
         \end{cases}\\
      &   
         \alpha^\sss_{t,(x,y)} = \,\,\,\,\,\, C\pfrac{1+(x+t)^2}{1+x^2}^{(1+y^2)\xi^\sss} (1+x^2)^{(1+y^2)\eps}, \,\,\,\, t\ge0\\
      &   
         \alpha^\uu_{t,(x,y)} = \,\,\,\,\,\, C\pfrac{1+(x+t)^2}{1+x^2}^{(1+y^2)\xi^\uu} (1+x^2)^{(1+y^2)\eps}, \,\,\,\, t \le 0.   \end{split}
\]
Notice that $d_{\lambda^\ell}(x,y)=\frac\partial{\partial x}\lambda^\ell(x,y)$.

\end{example}

\section{Discrete-time examples}\label{section:examples:disc}
In this section we assume $\T=\Z$ and provide some corollaries to Theorem~\ref{thm:global:disc}. 
Let $X$ be a Banach space and let $\Sigma\equiv(\W,\mathcal{F},\mathbb{P},\theta)$ be a measure-preserving dynamical system. Throughout this subsection we consider a real number $\delta \in\ ]0,1/6[$ and a random variable $G \colon \W \to\ ]0,+\infty[$ such that for all $\w \in \W$ we have $$\sum_{k=-\infty}^{+\infty} G(\theta^k\w) \le 1.$$

\subsection{Tempered exponential trichotomies}

 Consider  $\theta$-invariant random variables
   \begin{equation*}
      \lbd^\cs, \lbd^\cu, \lbd^\sss, \lbd^\uu \colon \W \to \R.
   \end{equation*}
   We say that a measurable linear RDS $\Phi$ on $X$ over $\Sigma$ exhibits an \emph{exponential trichotomy} if it admits a generalized trichotomy with bounds
     \begin{align*}
      &    
         \alpha^\cc_{n,\w} =
         \begin{cases}
            K(\w) \e^{\lbd^\cs(\w)n}, \,\, n\ge0\\
            K(\w) \e^{\lbd^\cu(\w)n}, \,\, n\le0
         \end{cases}\\
      &   
         \alpha^\sss_{n,\w} = \,\,\,\, K(\w) \e^{\lbd^\sss(\w)n}, \,\, n\ge0\\
      &   
         \alpha^\uu_{n,\w} = \,\,\,\, K(\w) \e^{\lbd^\uu(\w)n}, \,\, n \le 0
	\end{align*}
   for some random variable $K:\W \to [1,+\infty[$. If the  random variable $K$ is \emph{tempered} we say that $\Phi$ exhibits an \emph{tempered exponential trichotomy}. Notice that  in the discrete-time case the condition~\eqref{eq:tempered} is equivalent to
   \begin{equation*}
      \lim_{n\to\pm\infty}\frac1{\abs{n}} \log K(\theta^n\w)=0
      \text{ for all $\w\in\W$}.
   \end{equation*}

\begin{corollary}
   Let $\Phi$ be a measurable linear RDS exhibiting a tempered exponential trichotomy such that, for all $\w\in\W$, satisfies
       \begin{equation*}
      \lbd^\cu(\w)>\lbd^\sss(\w) 
      \quad\text{ and }\quad
      \lbd^\cs(\w)<\lbd^\uu(\w)
   \end{equation*}
   and let $f \in \sF$. Consider a $\theta$-invariant random variable $\gamma(\w)>0$ satisfying  for all $\w \in \W$
   \begin{equation*}
     a(\w):=  \lbd^\cu(\w)-\lbd^\sss(\w)-\gamma(\w)>0
       \,\text{ and } \, 
       b(\w):=\lbd^\uu(\w)-\lbd^\cs(\w)-\gamma(\w)>0.
   \end{equation*}
   If 
   \begin{align*}
      & \Lip(f_\w)\\
      & \le \frac{\delta}{K(\theta\w)}
         \mins{\e^{\min\{\lbd^\cu(\w),\lbd^\cs(\w)\}}G(\w),
         \e^{\lbd^\uu(\w)}\frac{\e^{a(\w)}-1}{\Lambda_{K,\gamma(\w),\w}},
         \e^{\lbd^\sss(\w)}\frac{1-\e^{-b(\w)}}{\Lambda_{K,\gamma(\w),\w}}}
   \end{align*}
for all $\w\in\W$ then the same conclusions of Theorem~\ref{thm:global} hold.
\end{corollary}

\subsection{$\psi$-trichotomies}

Consider measurable functions
      $$ \psi^{\cs},  \psi^{\cu}, \psi^{\sss},   \psi^{\uu}\colon \Z \x \W \to ]0,+\infty[$$
   such that for $\ell\in\{\cs, \cu, \sss, \uu\}$ we have
   \begin{equation*}
      \psi^\ell(t+s,\w) = \psi^\ell(t,\theta^s \w) \psi^\ell (s,\w)
   \end{equation*}
   for all $t,s \in \Z$ and all $\w \in \W$.   A \emph{$\psi$-trichotomy} is a generalized trichotomy with bounds
   \begin{equation*}
   \begin{split}
      &    
         \alpha^\cc_{n,\w} =
         \begin{cases}
            K(\w) \psi^\cs(n,\w), \,\, t\ge0\\
            K(\w) \psi^\cu(n,\w), \,\, t\le0
         \end{cases}\\
      &   
         \alpha^\sss_{t,\w} = \,\,\,\, K(\w) \psi^\sss(n,\w), \,\, t\ge0\\
      &   
         \alpha^\uu_{t,\w} = \,\,\,\, K(\w) \psi^\uu(n,\w), \,\, t \le 0   \end{split}
   \end{equation*}
   where $K\colon\W\to[1,+\infty[$ is a random variable.  We notice that, as in the continuous-time case, we may consider different growth rates along the  {central directions} $E_\w^\cc$, depending if we are looking to the \emph{future} ($n\to+\infty$) or to the \emph{past} ($n\to-\infty$).

\begin{corollary}
   Let $\Phi$ be a measurable linear RDS exhibiting  a $\psi$-trichotomy such that
   \begin{equation}\label{eq:psi<K<psi:disc}
      \dfrac{\psi^\cs(1,\w)}{\psi^\uu(1,\w)} < \dfrac{K(\theta \w)}{K(\w)}
      < \dfrac{\psi^\cu(1,\w)}{\psi^\sss(1,\w)}.
   \end{equation}
   Let $f \in \sF$ be such that
   \begin{equation*}
      \Lip(f_\w)
      \le \dfrac{\delta}{K(\theta\w)} \mins{\psi^\cs(1,\w)G(\w),\psi^\cu(1,\w) G(\w),a(\w),b(\w)},
   \end{equation*}
   where
   \begin{equation*}
      a(\w)= \dfrac{\psi^\uu(1,\w)}{K(\w)} - \dfrac{\psi^\cs(1,\w)}{K(\theta \w)}
      \,\,
      \text{ and } 
      \,\,
      b(\w)=\dfrac{\psi^\cu(1,\w)}{K(\theta\w)} - \dfrac{\psi^\sss(1,\w)}{K(\w)}.
   \end{equation*}
   If
   \begin{equation*}
      \dlim_{n \to -\infty} K(\theta^n \w) \psi^\sss(-n,\theta^n\w)\psi^\cu(n,\w) =\dlim_{n \to +\infty} K(\theta^n \w) \psi^\uu(-n,\theta^n\w)\psi^\cs(n,\w) = 0
   \end{equation*}
   for all $\w \in \W$, then the same conclusion of Theorem~\ref{thm:global:disc} holds.
\end{corollary}

\begin{proof}
   We will check that we are in conditions to apply Theorem~\ref{thm:global:disc}. Notice that from~\eqref{eq:psi<K<psi:disc} we conclude $a(\w), b(\w)>0$. We have
	\begin{align*}
         \sgm_\w^-
      & 
         = \sup\limits_{n\in\N}\dfrac{1}{\psi^\cu(-n,\w)}\dsum_{k=-n}^{-1}
         K(\theta^{k+1} \w) \psi^\cu(-n-k-1,\theta^{k+1}\w) \Lip(f_{\theta^k \w})\psi^\cs(k,\w)   
         \\
      & 
         \le \delta \dsum_{k=-\infty}^{+\infty} G(\theta^k\w) \le \delta,
	\end{align*}
   and, similarly, $\sgm_\w^+ \leq\delta$. Thus $\sgm\leq\delta$. 
   Moreover,
   \begin{align*}
         \tau_\w^+
      &
         =\dsum_{k=0}^{+\infty} K(\theta^{k+1}\w)\psi^\uu(-k-1,\theta^{k+1}\w)\Lip(f_{\theta^k\w}) K(\w)\psi^\cs(k,\w)\\
      &
         \leq\delta K(\w)\dsum_{k=0}^{+\infty} \left[\dfrac{\psi^\uu(-k,\theta^k\w)\psi^\cs(k,\w))}{K(\theta^k\w)} -\dfrac{\psi^\uu(-(k+1),\theta^{k+1}\w)\psi^\cs(k+1,\w))}{K(\theta^{k+1}\w)} \right]\\     
      &
         \leq \delta K(\w)\left(\dfrac{1}{K(\w)}-\lim_{k\to+\infty}\dfrac{\psi^\uu(-k,\theta^k\w)\psi^\cs(k,\w))}{K(\theta^k\w)}\right)\\
      &
         = \delta.
   \end{align*}
   Similarly we get $\tau_\w^-\leq \delta$.
   Therefore $\sigma+\tau \le 3\delta < 1/2$. 
\end{proof}

In the following we consider particular $\psi$-trichotomies.

\subsubsection{Summable exponential trichotomies}

We start by considering the integral (or summable) exponential trichotomies, which are  a generalization of  the exponential trichotomies and can be seen as the discrete counterpart of those in Section~\ref{sec:ex:integral}.
 Let
   \begin{equation*}
      \lbd^\cs, \lbd^\cu, \lbd^\sss, \lbd^\uu \colon \W \to \R
   \end{equation*}
   be random variables and set For all $\ell \in\{{\cs}, \cu, {\sss}, {\uu}\}$ we set
   \begin{equation*}
      S^\ell(n,\w)
      = 
      \begin{cases}
         \lbd^\ell(\w)+\cdots+\lbd^\ell(\theta^{n-1} \w), 
            & \ n\geq1\\
         0, 
            & \ n=0\\
         -\lbd^\ell(\theta^{n}\w)-\cdots -\lbd^\ell(\theta^{-1}\w), 
            & \ n\leq-1
      \end{cases}
   \end{equation*}
   An \emph{summable exponential trichotomy} is a $\psi$-trichotomy with
   \[
   \psi^\ell(t,\w)=\e^{S^\ell(n,\w)}
   \]
for all $\ell\in\set{\cs, \cu, \sss, \uu}$, i.e., is a generalized trichotomy with bounds
   \begin{align*}
      &    
         \alpha^\cc_{n,\w} =
         \begin{cases}
            K(\w) \e^{S^\cs(n,\w)}, \,\, n\ge0\\
            K(\w) \e^{S^\cu(n,\w)}, \,\, n\le0
         \end{cases}\\
      &   
         \alpha^\sss_{n,\w} = \,\,\,\, K(\w) \e^{S^\sss(n,\w)}, \,\, n\ge0\\
      &   
         \alpha^\uu_{n,\w} = \,\,\,\, K(\w) \e^{S^\uu(n,\w)}, \,\, n \le 0
	\end{align*}
      for some tempered random variable $K:\W \to [1,+\infty[$.

\begin{corollary}
   Let $\Phi$ be a measurable linear RDS exhibiting  a summable exponential trichotomy such that
   \begin{equation*}
      \dfrac{\e^{\lbd^\cs(\w)}}{\e^{\lbd^\uu(\w)}} 
      < \dfrac{K(\theta\w)}{K(\w)}
      < \dfrac{\e^{\lbd^\cu(\w)}}{\e^{\lbd^\sss(\w)}}.
   \end{equation*}
   Let $f \in \sF$ be such that
   \begin{equation*}
      \Lip(f_\w)
      \le \dfrac{\delta}{K(\theta\w)} \mins{\e^{\lbd^\cs(\w)}G(\w),\e^{\lbd^\cu(\w)} G(\w),a(\w),b(\w)},
   \end{equation*}
   where
     \[ a(\w)=      \dfrac{\e^{\lbd^\uu(\w)}}{K(\w)} - \dfrac{\e^{\lbd^\cs(\w)}}{K(\theta \w)}
    \,\,
    \text{ and }
\,\,
      b(\w)=\dfrac{\e^{\lbd^\cu(\w)}}{K(\theta\w)} - \dfrac{\e^{\lbd^\sss(\w)}}{K(\w)}.
   \end{equation*}
      If
   \begin{equation*}
      \dlim_{n \to -\infty} K(\theta^n \w) \e^{S^\sss(-n,\theta^n\w)+S^\cu(n,\w)} =\dlim_{n \to +\infty} K(\theta^n \w) \e^{S^\uu(-n,\theta^n\w)+S^\cs(n,\w)} = 0
   \end{equation*}
   for all $\w \in \W$, then the same conclusion of Theorem~\ref{thm:global:disc} holds.
\end{corollary}

\subsubsection{Non exponential trichotomies} 
We provide a particular type of $\psi$-trichotomies that can be easily handled to construct trichotomies beyhond the exponential bounds in the discrete-time scenario. Consider a $\psi$-tri\-chot\-o\-my with
   \[
   \psi^\ell(n,\w)=\dfrac{\lbd^\ell(\w)}{\lbd^\ell(\theta
^n\w)}
   \]
for all $\ell\in\set{\cs, \cu, \sss, \uu}$, \emph{i.e.}, is a generalized trichotomy with bounds
   \begin{equation}\label{eq:nonexp:bounds:disc:1}
   \begin{split}
      &    
         \alpha^\cc_{n,\w} =
         \begin{cases}
             K(\w) \dfrac{\lbd^\cs(\w)}{\lbd^\cs(\theta^n \w)}, \,\, n\ge0\\[4mm] 
            K(\w) \dfrac{\lbd^\cu(\w)}{\lbd^\cu(\theta^n \w)}, \,\, n\le0
         \end{cases}\\
      &   
         \alpha^\sss_{n,\w} = \,\,\,\, K(\w) \frac{\lbd^\sss(\w)}{\lbd^\sss(\theta^n \w)}, \,\, n\ge0\\
      &   
         \alpha^\uu_{n,\w} = \,\,\,\, K(\w) \frac{\lbd^\uu(\w)}{\lbd^\uu(\theta^n \w)}, \,\, n \le 0.   \end{split}
   \end{equation}

For future use let us define
 \begin{equation*}
      a(\w)=      \dfrac{\lbd^\uu(\w)}{\lbd^\uu(\theta
\w)K(\w)} - \dfrac{\lbd^\cs(\w)}{\lbd^\cs(\theta
\w)K(\theta\w)}
   \end{equation*}
   and
    \begin{equation*}
      b(\w)=\dfrac{\lbd^\cu(\w)}{\lbd^\cu(\theta
\w)K(\theta\w)} - \dfrac{\lbd^\sss(\w)}{\lbd^\sss(\theta
\w)K(\w)}.
   \end{equation*}

\begin{corollary}
   Let $\Phi$ be a measurable linear RDS exhibiting  an $\alpha$-tri\-chot\-o\-my with bounds~\eqref{eq:nonexp:bounds:disc:1} and such that
   \begin{equation*}
    \dfrac{\lbd^\cs(\w)\lbd^\uu(\theta
\w)}{\lbd^\cs(\theta
\w)\lbd^\uu(\w)} < \dfrac{K(\theta\w)}{K(\w)}<\dfrac{\lbd^\cu(\w)\lbd^\sss(\theta
\w)}{\lbd^\cu(\theta
\w)\lbd^\sss(\w)}.
   \end{equation*}
    Let $f \in \sF$ be such that
      \begin{equation*}
      \Lip(f_\w)
      \le \dfrac{\delta}{K(\theta\w)} \mins{\dfrac{\lbd^\cs(\w)}{\lbd^\cs(\theta
\w)}G(\w),\dfrac{\lbd^\cu(\w)}{\lbd^\cu(\theta
\w)}G(\w),a(\w),b(\w)}.
   \end{equation*}
     If
   \begin{equation*}
      \dlim_{n \to -\infty} \dfrac{K(\theta^n \w)\lbd^\sss(\theta^n \w)}{\lbd^\cu(\theta^n \w)}=\dlim_{n \to +\infty} \dfrac{K(\theta^n \w)\lbd^\uu(\theta^n \w)}{\lbd^\cs(\theta^n \w)} = 0
   \end{equation*}
   for all $\w \in \W$, then the same conclusion of Theorem~\ref{thm:global:disc} holds.
\end{corollary}

We may consider Example~\ref{example:Nonexponential trichotomy} with $\T=\Z$ to get an apllication of this result in a nonexponential trichotomy situation.

\section*{Acknowledgements}
This work was partially supported by Funda\c c\~ao para a Ci\^encia e Tecnologia through Centro de Matem\'atica e Aplica\c c\~oes da Universidade da Beira Interior (CMA-UBI), project UIBD/00212/2020.
\renewcommand{\baselinestretch}{1}

\end{document}